\documentclass{amsart}
\usepackage{url}
\usepackage{hyperref}
\usepackage{breakurl}
\usepackage{graphicx}
\usepackage{fancyhdr}
\usepackage[all]{xy}
\usepackage{mathrsfs}
\usepackage[mathscr]{eucal}
\usepackage{graphics, graphpap}
\usepackage{array, tabularx, longtable}
\usepackage{amsthm, amsmath, amscd, amssymb,amsfonts}
\theoremstyle{plain}
\usepackage[top=3.0cm, bottom=3.0cm, left=3.0cm, right=3.0cm]{geometry}
\newtheorem{thm}{Theorem}[section]
\newtheorem{theorem}[thm]{Theorem}

\newtheorem{corollary}[thm]{Corollary}
\newtheorem{lemma}[thm]{Lemma}

\newtheorem{remark}[thm]{Remark}

\newtheorem {conjecture}[thm]{Conjecture}

\numberwithin{equation}{section}
\title{Right unimodal and bimodal singularities in positive characteristic}      
\author{Nguyen Hong Duc}
\address{Institute of Mathematics, Vietnam Academy of Science and Technology\newline \indent 18 Hoang Quoc Viet Road, Cau Giay District, \newline \indent  10307, Hanoi.} 
\email{nhduc@math.ac.vn}
\begin{document}
\begin{abstract}
The problem of classification of real and complex singularities was initiated by Arnol'd in the sixties who classified simple, unimodal and bimodal w.r.t. right equivalence. The classification of right simple singularities in positive characteristic was achieved by Greuel and the author in 2014. In the present paper we classify right unimodal and bimodal singularities in positive characteristic by giving explicit normal forms. Moreover we completely determine all possible adjacencies of simple, unimodal and bimodal singularities. As an application we prove that, for singularities of right modality at most 2, the $\mu$-constant stratum is smooth and its dimension is equal to the right modality. In contrast to the complex analytic case, there are, for any positive characteristic, only finitely many 1-dimensional (resp. 2-dimensional) families of right class of unimodal (resp. bimodal) singularities. We show that for fixed characteristic $p>0$ of the ground field, the Milnor number of $f$ satisfies $\mu(f)\leq 4p$, if the right modality of $f$ is at most 2. 
\end{abstract}
\maketitle
\section{Introduction}
We classify hypersurface singularities $f\in K[[x_1,\ldots,x_n]]$ which are unimodal and bimodal w.r.t. right equivalence, where $K$ is an algebraically closed field of positive characteristic. That is, the singularities have modality 1 resp. 2 up to the change of coordinates (or right equivalence, see Section \ref{sec21}). The notion of modality was introduced by Arnol'd in the seventies \cite{Arn72}, \cite{Arn73}, \cite{Arn76} into singularity theory for real and complex singularities. He classified simple, unimodal and bimodal hypersurface singularities w.r.t. right equivalence. He showed that the simple singularities are exactly the $ADE$-singularities, i.e. the two infinite series $A_k, k\geq 1$, $D_k, k\geq 4$, and the three exceptional singularities $E_6, E_7, E_8$. The right simple singularities in positive characteristic were recently classified by Greuel and the author in \cite{GN14}. 

The main result of the present paper is the classification of unimodal and bimodal singularities w.r.t. right equivalence with tables of normal forms. Recall that a normal form is a {\em modular family} $F({\bf x},t)\in \mathcal{O}(T)[[{\bf x}]]$ (see $\S$\ref{sec2}), i.e. for each $t\in T$ there are only finitely many $ t'\in T$ such that $f_{t'}\sim_r f_{t}$. Notice that, if $F({\bf x},t)$ is a normal form, then $\mathrm{mod}(F({\bf x},t))\geq \dim T$ for all $t\in T$. 
 Our lists of normal forms for unimodal and bimodal singularities are given in $\S \ref{sec3}$. In contrast to the complex analytic case, there exist only finitely many $r$-dimensional normal forms for $r$-modal singularities for $r\leq 2$. Moreover, we obtain that for a singularity $f$ with modality at most 2, it Milnor number is bounded by a function of the characteristic. Precisely, we show in Corollary \ref{coro31} that 
$$\mu(f)\leq 4p.$$
As an application of the classification, we obtain that if $f$ is simple, unimodal or bimodal singularity, then its $\mu$-constant stratum is smooth. Consequently, we prove that right modality and proper modality coincide (see $\S$\ref{sec2} for definitions). We conjecture that the equality holds in general, see Conjecture \ref{conj24}. 

Section \ref{sec4} is an outline of the proofs of the main results. The proofs are organized in the form of a singularity determinator, finding for every given singularities its place in the list of $\S$ \ref{sec3}, similar to Arnold's classification in \cite{Arn76}. We present an algorithm for determining the right class of a singularity in the form of 152 theorems. The main results are proved in Section \ref{sec5}. 

Note that, for contact equivalence and for $K=\mathbb{C}$, it was proved by Giusti in \cite{Giu77} that ADE-singularities are contact simple. The classification of contact unimodal singularities was achieved by Wall in \cite{Wal83}. Greuel and Kr\"oning showed in \cite{GK90} that the contact simple singularities over a field of positive characteristic are again exactly the $ADE$-singularities or the rational double points of Artin's list \cite{Art77}.
\subsection*{Acknowledgement} 
The author would like to thank Gert-Martin Greuel for valuable discussions.
This research project was partially supported by Vietnam National Foundation for Science and Technology Development(NAFOSTED) grant 101.04-2014.23, and and the Oberwolfach Leibniz Fellows programme of the Mathematisches Forschungsinstitut Oberwolfach (Germany).
\section{Modality}\label{sec2}
Modality was introduced by Arnol'd in connection with the classification of singularities of functions under right equivalence. It has been generalized to arbitrary actions of algebraic groups by Vinberg [18]. Wall [20] described two possible generalizations for use in other classification problems in singularity theory. Both geeralization are developed in detail by Greuel and the author (\cite{GN14}) for any characteristic and it was proved that they coincide. 

\subsection{Right modality}\label{sec21}
Consider an action of algebraic group $G$ on a variety $X$ (over a given algebraically closed field $K$) and a {\em Rosenlicht stratification} $\{(X_i,p_i), i=1,\ldots, s\}$ of $X$ w.r.t. $G$. That is, a stratification $X=\cup_{i=1}^sX_i$, where the stratum $X_i$ is a locally closed $G$-invariant subvariety of $X$ such that the projection $p_i:X_i\to X_i/G$ is a geometric quotient. For each open subset $U\subset X$ the modality of $U$, $G\text{-}\mathrm{mod}(U)$, is the maximal dimension of the images of $U \cap X_i$ in $X_i/G$. The modality $G\text{-}\mathrm{mod}(x)$ of a point $x \in X$ is the minimum of $G\text{-}\mathrm{mod}(U)$ over all open neighbourhoods $U$ of $x$.

Let $K[[{\bf x}]]=K[[x_1,\ldots,x_n]]$ the formal power series ring and let the right group, $\mathcal R:=Aut(K[[{\bf x]}])$, act on $K[[{\bf x}]]$ by $(\Phi,f) \mapsto \Phi(f)$. Two elements $f,g\in K[[{\bf x}]]$ are called {\em right equivalent}, $f\sim_r g$, if they belong to the same $\mathcal{R}$-orbit, or equivalently, there exists a coordinate change $\Phi\in Aut(K[[{\bf x]}])$ such that $g=\Phi(f)$.

Recall that for $f\in\mathfrak{m}\subset K[[{\bf x}]]$, $\mu(f):=\dim K[[{\bf x}]]/j(f),\ j(f)=\langle f_{x_1}, \ldots, f_{x_n}\rangle$, denotes the {\em Milnor number} of $f$ and that $f$ is {\em isolated} if $\mu(f)<\infty$. The $k$-jet of $f$, $j^k(f)$, is the image of $f$ in the jet space $J_k:=\mathfrak{m}/\mathfrak{m}^{k+1}$. We call $f$ to be right $k$-determined if each singularity having the same $k$-jet with $f$, is right equivalent to $f$. A number $k$ is called {\em right sufficiently large} for $f$, if there exists a neighbourhood $U$ of the $j^k f$ in $J_k$ such that every $g\in K[[{\bf x}]]$ with $j^kg\in U$ is right $k$-determined. The {\em right modality} of $f$, $\mathcal R\text{-}\mathrm{mod}(f)$, is defined to be the $\mathcal R_k\text{-}\mathrm{modality}$ of $j^k f$ in $J_k$ with $k$ right sufficiently large for $f$, where $\mathcal R_k$ the $k$-jet of $\mathcal R$. A singularity $f\in K[[{\bf x}]]$ is called ({\em right}) {\em simple, uni-modal, bi-modal} and {\em $r$-modal} if its (right) modality is equal to 0,1,2 and $r$ respectively. These notions are independent of the right sufficiently large $k$.

The second description is in relation with versal or complete deformation. Let $T$ be an affine variety with its algebra of global section $\mathcal{O}(T)$. Then a family $f_{t}({\bf x}):=F({\bf x},t)\in \mathcal{O}(T)[[{\bf x}]]$ is called an {\em unfolding} (deformation with trivial section) of $f$ over a pointed space $T, t_0$ if $F({\bf x},t_0)=f\text{ and }f_{t}\in \mathfrak{m}$ for all $t\in T$. A {\em semiuniversal unfolding} is given by
$$F({\bf x},\lambda):=f({\bf x})+\sum_{i=1}^{N}\lambda_i {\bf x}^{\alpha_i},$$
with $\lambda=(\lambda_1,\ldots,\lambda_N)$ is the coordinate of $\lambda\in \mathbb{A}^N$ and $\{{\bf x}^{\alpha_1},\ldots,{\bf x}^{\alpha_N}\}$ is a basis of $\mathfrak{m}/\mathfrak{m}\cdot j(f)$. Note that from the exact sequence
$$0\rightarrow j(f)/\mathfrak{m}\cdot j(f) \rightarrow \mathfrak{m}/\mathfrak{m}\cdot j(f) \rightarrow \mathfrak{m}/j(f)\rightarrow 0$$ 
we get $N=\mu+n-1$. Since $\mathfrak{m}\cdot j(f)$ is the tangent space of the orbit of the right group $\mathcal R$ at $f$ (\cite{BGM12}), $N$ is the codimension of the orbit in $\mathfrak{m}$.

An unfolding $F({\bf x},t)$ over $T,{t}_0$ is called {\em right complete} if any unfolding $H({\bf x},s)$ over $S,s_0$ is isomorphic to a pullback of $F$ after passing to some \'etale neighbourhood of $S,s_0$, see \cite{GN14}. An important property of the complete unfoldings is that they are sufficient to determine the modality, i.e. if $F$ is a complete unfolding of $f$, then modality of $f$ w.r.t. $F$ (\cite[Def. 2.5]{GN14}) equals to modality of $f$, see \cite[Prop. 2.12(ii)]{GN14}. A semiuniversal unfolding of an isolated hypersurface singularity is right complete (see \cite{KaS72} for the analytic case and \cite{GN14} for the general case). Consequently, we may define modality of $f$ as follows: {\em ``Let $f_\lambda$ be a semiuniversal unfolding of $f$ over $\mathbb{A}^N,0$. If the set of singularities $f_\lambda\in K[[{\bf x}]]$ ($\lambda$ in some Zariski neighbourhood of $0\in \mathbb{A}^N$) falls into finitely many families of right classes, each depending on $r$ parameters (at most) then $f$ is right (resp. contact) $r$-modal (at most).''}
\begin{remark}{\rm
(1) For convergent power series over the complex numbers it does not make any difference wheter we consider the semiuniversal deformation (without section) given by the Milnor algebra $\mathbb{C}\{x_1,\ldots,x_n\}/j(f)$ or the semiuniversal deformation with section given by $\mathfrak{m}/\mathfrak{m}\cdot j(f)$. However, in positive characteristic we have to consider the latter (cf. \cite{GN14} and \ref{sec22}).

(2) The difference between the classical versal and our complete deformation is twofold. First, we consider deformations over algebraic varieties and not just of the spectrum of a complete local ring (as for versal deformations). Second, we do not require the lifting property for induced deformations over small extensions (cf. \cite[Ch.2]{GLS06}). 
}\end{remark}
\subsection{Proper modality}\label{sec22}
In \cite{Gab74} Gabrielov showed in the complex analytic case that the right modality is equal to the dimension of the $\mu$-constant stratum in a semi-universal deformation of $f$. This is not true in positive characteristic since $f=x^2+y^4\in A_3\subset K[[x,y]]$ with $\mathrm{char}(K)=3$, is unimodal, but the dimension of the stratum $\mu=3$ into the semiuniversal deformation $f+a_0+a_1y+a_2y^2$, is equal to $0$. In positive characteristic we need to consider deformations with section. Let $f_{\lambda}({\bf x}):=F({\bf x},\lambda)$ be the semiuniversal unfolding of $f$ with trivial section over affine variety $\mathbb{A}^N,0$ with $N=\mu+n-1$ as above. We define the {\em proper modality} of $f$, denoted by $\mathrm{pmod}(f)$, to be the dimension at $0$ of the $\mu$-constant stratum in $\mathbb{A}^N$:
$$\Delta_{\mu}:=\{\lambda\in \mathbb{A}^N\ |\ \mu(f_{\lambda})=\mu\}.$$ 
\begin{conjecture}\label{conj24}
$\mathrm{pmod}(f)=\mathrm{rmod}(f)$.
\end{conjecture}
See Corollary \ref{coro33} for a partial result of the conjecture. Namely, if $\mathrm{rmod}(f)\leq 2$ then $\mathrm{pmod}(f)=\mathrm{rmod}(f)$.

\section{Right unimodal, bimodal singularities and adjacency diagrams}\label{sec3}
In this section we present the result of our classification and adjacency diagrams of simple, unimodal and bimodal singularities. 
\subsection{Right unimodal singularities}
\begin{theorem}\label{thm21}
Let $p=\mathrm{char}(K)>2$. A hypersurface singularity $f\in\mathfrak{m}^2$ is right unimodal if and only if it is right equivalent to one of the following forms:

I. {$\bf n=1$} $(f\in K[[x]])$. The classification is given in Table \ref{table1}.

II. {$\bf n=2$} $(f\in K[[x,y]])$. The classification is given in Table \ref{table3}.

III. {$\bf n=3$} $(f\in K[[x,y,z]])$. The classification is given in Table \ref{table4}.

IV. {$\bf n>3$}. The classification is given in Table \ref{table5}.
\end{theorem}
\begin{theorem}\label{thm22}
Let $p=\mathrm{char}(K)=2$. A hypersurface singularity $f\in\mathfrak{m}^2$ is right unimodal if and only if $n$ is odd and $f$ is right equivalent to one of singularities in the Table \ref{table5.1}.
\end{theorem}

\begin{table}[h]
\begin{tabular}[20pt]{|c|l l l|}
\hline 
Name& Normal form&Conditions&$\mu$ \\
\hline
$\mathrm{A}_{k}$&$x^p+a x^{k+1}$& $p\leq k\leq 2p-2$& $k$\\
\hline
\end{tabular}
\newline
\caption{}\label{table1}
\end{table}


\begin{table}[h]
\begin{tabular}[20pt]{|c|l l l|}
\hline 
Name& Normal form&Conditions&$\mu$ \\
\hline
$\mathrm{A}_{k}$&$x^2+ay^{p}+y^{k+1}$&$p\leq k\leq 2p-2$&$k$\\
\hline
$\mathrm{D}_{p}$&$x^2y+y^{p-1}$&$3< p$&$p$\\
\hline
$\mathrm{D}_{k}$&$x^2y+ay^{p}+y^{k-1}$&$3\leq p< k-1\leq 2p-2$&$k$\\
\hline
$\mathrm{E}_{12}$& $x^3+y^7+axy^5$&$7<p$&$12$\\
\hline
$\mathrm{E}_{13}$&  $x^3+xy^5+ay^8$&$7<p$&$13$\\
\hline
$\mathrm{E}_{14}$& $x^3+y^8+axy^6$&$7<p$&$14$\\
\hline
$\mathrm{J}_{10}=\mathrm{J}_{2,0}$& $x^3+y^6+ax^2y^2$&$5<p$&$10$\\
\hline
$\mathrm{J}_{2,q}$& $x^3+ax^2y^2+y^{6+q}$&$6<6+q<p$&$q+10$\\
\hline
$\mathrm{W}_{12}$& $x^4+y^5+ax^2y^3$&$p>5$&$12$\\
\hline
$\mathrm{W}_{13}$& $x^4+xy^4+ay^6$&$p>5$&$13$\\
\hline
$\mathrm{X}_{9}=\mathrm{X}_{1,0}$& $x^4+y^4+ax^2y^2$&$3<p$&$9$\\
\hline
$\mathrm{X}_{1,q}$& $x^4+x^2y^2+ay^{4+q}$&$4<4+q<p$&$q+9$\\
\hline
$\mathrm{Y}_{r,s}$& $x^{4+r}+ax^2y^2+y^{4+s}$&$4<4+r\leq 4+s<p$&$9+r+s$\\
\hline
$\mathrm{Z}_{11}$& $x^3y+y^5+axy^4$&$5<p$&$11$\\
\hline
$\mathrm{Z}_{12}$& $x^3y+xy^4+ax^2y^3$&$5<p$&$12$\\
\hline
$\mathrm{Z}_{13}$& $x^3y+y^6+axy^5$&$5<p$&$13$\\
\hline
\end{tabular}
\newline
\caption{}\label{table3}
\end{table}


\begin{table}[ht]
\begin{tabular}[20pt]{|c|l l l|}
\hline 
Name& Normal form&Conditions&$\mu$ \\
\hline
&$g(x,y)+z^2$&$g$ one of the series in Table \ref{table3}&$\mu(g)$\\
\hline
$\mathrm{P}_8=\mathrm{T}_{3,3,3}$&$x^3+y^{3}+z^{3}+axyz$&$3<p$&$8$\\
\hline
$\mathrm{Q}_{10}$&$x^3+y^{4}+yz^{2}+axy^3$&$3<p$&$10$\\
\hline
$\mathrm{Q}_{11}$&$x^3+yz^{2}+xz^{3}+az^5$&$3<p$&$11$\\
\hline
$\mathrm{Q}_{12}$&$x^3+y^{5}+yz^{2}+axy^4$&$5<p$&$12$\\
\hline
$\mathrm{S}_{11}$&$x^4+y^{2}z+xz^{2}+ax^3z$&$3<p$&$11$\\
\hline
$\mathrm{S}_{12}$&$x^2y+y^{2}z+xz^{3}+az^5$&$3<p$&$12$\\
\hline
$\mathrm{T}_{r,s,t}$&$x^r+y^{s}+z^{t}+axyz$&$3\leq r\leq s\leq t<p,\frac{1}{r}+\frac{1}{s}+\frac{1}{t}<1$,&$r+s+t-1$\\
\hline
$\mathrm{U}_{12}$&$x^3+y^{3}+z^{4}+axyz^2$&$3<p$&$12$\\
\hline
\end{tabular}
\newline
\caption{}\label{table4}
\end{table}


\begin{table}[ht]
\begin{tabular}[20pt]{|l| l|}
\hline 
\multicolumn{2}{|l|}{Normal form}\\
\hline 
$g(x_1,x_2,x_3)+x_4^2+\ldots+x_n^2$& $g$ is one of the singularities in Table \ref{table4}\\
\hline
\end{tabular}
\newline
\caption{}\label{table5}
\end{table}


\begin{table}[ht]
\begin{tabular}[20pt]{|c|l l l|}
\hline 
Name& Normal form&Conditions&$\mu$ \\
\hline
$\mathrm{A}_{2}$&$ax_1^2+ x_1^3+x_2x_3+\ldots+x_{n-1}x_n$&$a\in K$&$2$ \\
\hline
\end{tabular}
\newline
\caption{}\label{table5.1}
\end{table}
\newpage
\subsection{Right bimodal singularities}
\begin{theorem}\label{thm23}
Let $p=\mathrm{char}(K)>2$. A hypersurface singularity $f\in\mathfrak{m}^2$ is right bimodal if and only if it is right equivalent to one of the following forms

I. {$\bf n=1$} $(f\in K[[x]])$. The list is given in Table \ref{table5.2}.

II. {$\bf n=2$} $(f\in K[[x,y]])$. The list is given in Table \ref{table6}.

III. {$\bf n=3$} $(f\in K[[x,y,z]])$. The list is given in Table \ref{table7}.

IV. {$\bf n>3$}. The list is given in Table \ref{table8}.
\end{theorem}
\begin{theorem}\label{thm24}
Let $p=\mathrm{char}(K)=2$. A hypersurface singularity $f\in\mathfrak{m}^2$ is right bimodal if and only if it is right equivalent to one of the following forms

I. {$\bf n$ odd}: The list is given in the Table \ref{table8.1}.

II. {$\bf n$ even}: The list is given in the Table \ref{table9}.
\end{theorem}
\begin{table}[ht]
\begin{tabular}[20pt]{|c|l l l|}
\hline 
Name& Normal form&Conditions&$\mu$ \\
\hline
$\mathrm{A}_{k}$&$a_1x^p+ a_2x^{2p}+ x^{k+1}$& $2p\leq k\leq 3p-2$&$k$\\
\hline
\end{tabular}
\newline
\caption{}\label{table5.2}
\end{table}

 
\begin{table}[ht]
\begin{tabular}[20pt]{|c|l l l|}
\hline 
Name& Normal form (${\bf a}=a_0+a_1y$) & Conditions & $\mu$ \\
\hline
$\mathrm{A}_{k}$&$x^2+a_1y^p+ a_2y^{2p}+ y^{k+1}$&$2p\leq k\leq 3p-2$& $k$\\
\hline
$\mathrm{D}_{2p}$&$x^2y+ay^p+y^{2p-1}$&$3\leq p$& $2p$\\
\hline
$\mathrm{D}_{k}$&$x^2y+a_1y^p+ a_2y^{2p}+y^{k-1}$&$2p< k-1\leq 3p-1$& $k$\\
\hline
$\mathrm{E}_{12}$& $x^3+ay^5+y^7+bxy^5$&$p=5$& $12$\\
\hline
$\mathrm{E}_{13}$& $x^3+xy^5+{\bf a}y^7$&$p=7$& $13$\\
\hline
$\mathrm{E}_{14}$& $x^3+y^{8}+ay^7+bxy^6$&$p=7$& $14$\\
\hline
$\mathrm{E}_{18}$& $x^3+y^{10}+{\bf a}xy^7$&$7<p$& $18$\\
\hline
$\mathrm{E}_{19}$&  $x^3+xy^7+{\bf a}y^{11}$&$7<p$& $19$\\
\hline
$\mathrm{E}_{20}$& $x^3+y^{11}+{\bf a}xy^8$&$11<p$& $20$\\
\hline
$\mathrm{J}_{2,0}$& $x^3+bx^2y^2+y^6+ay^5$&$4b^3+27\neq 0, p=5$& $10$\\
\hline
$\mathrm{J}_{2,q}$& $x^3+x^2y^2+ay^p+by^{6+q}$&$p<6+q<2p,b\neq 0, p\geq 5$& $q+10$\\
\hline
$\mathrm{J}_{3,0}$& $x^3+bx^2y^3+cxy^7+y^9$&$4b^3+27\neq 0, 7<p$& $16$\\
\hline
$\mathrm{J}_{3,q}$& $x^3+x^2y^3+{\bf a}y^{9+q}$&$a_0\neq 0, 9<9+q<p$& $q+16$\\
\hline
$\mathrm{W}_{17}$& $x^4+xy^{5}+{\bf a}y^7$&$7<p$& $17$\\
\hline
$\mathrm{W}_{18}$&  $x^4+y^7+{\bf a}x^2y^{4}$&$7<p$& $18$\\
\hline
$\mathrm{W}_{1,0}$& $x^4+{\bf a}x^2y^{3}+y^6$&$a_0^2\neq 4, 5<p$& $15$\\
\hline
$\mathrm{W}_{1,q}$& $x^4+x^2y^{3}+{\bf a}y^{6+q}$&$a_0\neq 0, 7\leq 6+q<p$& $q+15$\\
\hline
$\mathrm{W}^{\sharp}_{1,2q-1}$& $(x^2+y^{3})^2+{\bf a}xy^{4+q}$&$a_0\neq 0, 5\leq 4+q<p$& $2q+14$\\
\hline
$\mathrm{W}^{\sharp}_{1,2q}$& $(x^2+y^{3})^2+{\bf a}x^2y^{3+q}$&$a_0\neq 0, 4\leq 3+q<p>5$& $2q+15$\\
\hline
$\mathrm{Z}_{12}$& $x^3y+xy^4+ay^5+bx^2y^{3}$&$p=5$& $12$\\
\hline
$\mathrm{Z}_{13}$& $x^3y+y^6+ ay^{5}+bxy^5$&$p=5$& $13$\\
\hline
$\mathrm{Z}_{17}$& $x^3y+{\bf a}xy^{6}+y^8$&$7<p$& $17$\\
\hline
$\mathrm{Z}_{18}$& $x^3y+xy^6+{\bf a}y^9$&$7<p$& $18$\\
\hline
$\mathrm{Z}_{19}$& $x^3y+y^9+{\bf a}xy^7$&$7<p$& $19$\\
\hline
$\mathrm{Z}_{1,0}$& $x^3y+bx^2y^{3}+cxy^6+y^7$&$4b^3+27\neq 0, 7<p$& $15$\\
\hline
$\mathrm{Z}_{1,q}$& $x^3y+x^2y^{3}+{\bf a}y^{7+q}$&$a_0\neq 0, 7<7+q<p$& $q+15$\\
\hline
\end{tabular}
\newline
\caption{}\label{table6}
\end{table}
 
\begin{table}[ht]
\begin{tabular}[20pt]{|c|l l l|}
\hline 
Name& Normal form (${\bf a}=a_0+a_1y$) & Conditions & $\mu$ \\
\hline
&$g(x,y)+z^2$&$g$ one of the series in Table \ref{table6}& $\mu(g)$\\
\hline
$\mathrm{Q}_{16}$&$x^3+yz^{2}+y^7+{\bf a}xy^5$&$7<p$& $16$\\
\hline
$\mathrm{Q}_{17}$&$x^3+yz^{2}+y^7+{\bf a}y^8$&$7<p$& $17$\\
\hline
$\mathrm{Q}_{18}$&$x^3+yz^{2}+y^8+{\bf a}xy^6$&$7<p$& $18$\\
\hline
$\mathrm{Q}_{2,0}$&$x^3+yz^{2}+{\bf a}x^2y^2+xy^4$&$a_0^2\neq 4,3<p$ & $14$\\
\hline
$\mathrm{Q}_{2,q}$&$x^3+yz^{2}+x^2y^2+{\bf a}y^{6+q}$&$a_0\neq 0,7\leq 6+q<p$ & $q+14$\\
\hline
$\mathrm{S}_{16}$&$x^2z+yz^{2}+xy^{4}+{\bf a}y^6$&$5<p$ & $16$\\
\hline
$\mathrm{S}_{17}$&$x^2z+yz^{2}+y^{6}+{\bf a}zy^4$&$5<p$& $17$\\
\hline
$\mathrm{S}_{1,0}$&$x^2y+yz^{2}+y^5+{\bf a}zy^3$&$a_0^2\neq 4,3<p$& $14$\\
\hline
$\mathrm{S}_{1,q}$&$x^2y+yz^{2}+x^2y^2+{\bf a}y^{5+q}$&$a_0\neq 0,5<5+q<p$& $q+14$\\
\hline
$\mathrm{S}^{\sharp}_{1,2q-1}$&$x^2y+yz^{2}+zy^3+{\bf a}xy^{3+q}$&$a_0\neq 0,3<3+q<p$& $2q+13$\\
\hline
$\mathrm{S}^{\sharp}_{1,2q}$&$x^2y+yz^{2}+zy^3+{\bf a}x^2y^{2+q}$&$a_0\neq 0,3\leq 2+q<p$& $2q+14$\\
\hline
$\mathrm{T}_{r,s,t}$&$x^r+y^{s}+z^{t}+axyz+bz^{p}$&$3\leq r\leq s<p< t<2p$& $r+s+t-1$\\
\hline
$\mathrm{U}_{16}$&$x^3+xz^2+y^{5}+{\bf a}x^2y^2$&$5<p$& $16$\\
\hline
$\mathrm{U}_{1,0}$&$x^3+xz^2+xy^{3}+{\bf a}y^3z$&$a_0(a_0^2+1)\neq 0,5<p$& $14$\\
\hline
$\mathrm{U}_{1,2q-1}$&$x^3+xz^2+xy^{3}+{\bf a}y^{1+q}z^2$&$a_0\neq 0,2\leq 1+q<p>3$& $2q+13$\\
\hline
$\mathrm{U}_{1,2q}$&$x^3+xz^2+xy^{3}+{\bf a}y^{3+q}z$&$a_0\neq 0,3< 3+q<p$& $2q+14$\\
\hline
\end{tabular}
\newline
\caption{}\label{table7}
\end{table}


\begin{table}[ht]
\begin{tabular}[20pt]{|l| l|}
\hline 
\multicolumn{2}{|l|}{Normal form}\\
\hline 
$g(x_1,x_2,x_3)+x_4^2+\ldots+x_n^2$& $g$ is one of the singularities in Table \ref{table7} \\
\hline
\end{tabular}
\newline
\caption{}\label{table8}
\end{table}
\newpage


\begin{table}[ht]
\begin{tabular}[20pt]{|c|l l|}
\hline 
Name& \multicolumn{2}{l|}{Normal form}\\
\hline
$\mathrm{A}_{4}$&$a_1x_1^2+ a_2x_1^4+x_1^5+x_2x_3+\ldots+x_{n-1}x_n$&$a_1,a_2\in K$\\
\hline
\end{tabular}
\newline
\caption{}\label{table8.1}
\end{table}

\begin{table}[ht]
\begin{tabular}[20pt]{|c|ll|}
\hline 
Name& Normal form & $\mu$ \\
\hline
$\mathrm{D}_{4}$&$a_1x_1^2+ a_2x_2^2+x_1^3+x_2^3+x_3x_4+\ldots+x_{n-1}x_n$& $4$\\
$\mathrm{D}_{6}$&$a_1x_1^2+ a_2x_2^2+x_1^2x_2+x_1x_2^3+x_3x_4+\ldots+x_{n-1}x_n$& $6$\\
$\mathrm{E}_{7}$&$a_1x_1^2+ a_2x_2^2+x_1^3+x_1x_2^3+x_3x_4+\ldots+x_{n-1}x_n$& $7$\\
$\mathrm{E}_{8}$&$a_1x_1^2+ a_2x_2^2+x_1^3+x_2^5+x_3x_4+\ldots+x_{n-1}x_n$& $8$\\
\hline
\end{tabular}
\newline
\caption{}\label{table9}
\end{table}

\subsection{Adjacencies of simple, uni- and bi-modal singularities}
In the following we give diagrams of adjacencies for all class of singularities in Table 1--7. Moreover a singularity in these tables deforms only into classes listed in the diagrams. Recall that a class $\mathcal{D}$ of singularities is {\em adjacent} to class $\mathcal{C}$, $\mathcal{C} \leftarrow \mathcal{D}$, if every $f\in \mathcal{D}$ can be {\em deformed into} an element in $\mathcal{C}$ by a deformation. That is, there exists an unfolding $f_{t}$ of $f=f_{t_0}$ over an affine variety $T,t_0$ and a Zariski open subset $V\subset T$ such that $f_{t}\in \mathcal{C}$ for all $t\in V$.\\
\begin{theorem}\label{thm35}
Any singularity in Table 1--11 deforms only into singularities given in the adjacency diagrams 1--15:
\end{theorem}
Adjacency diagrams:
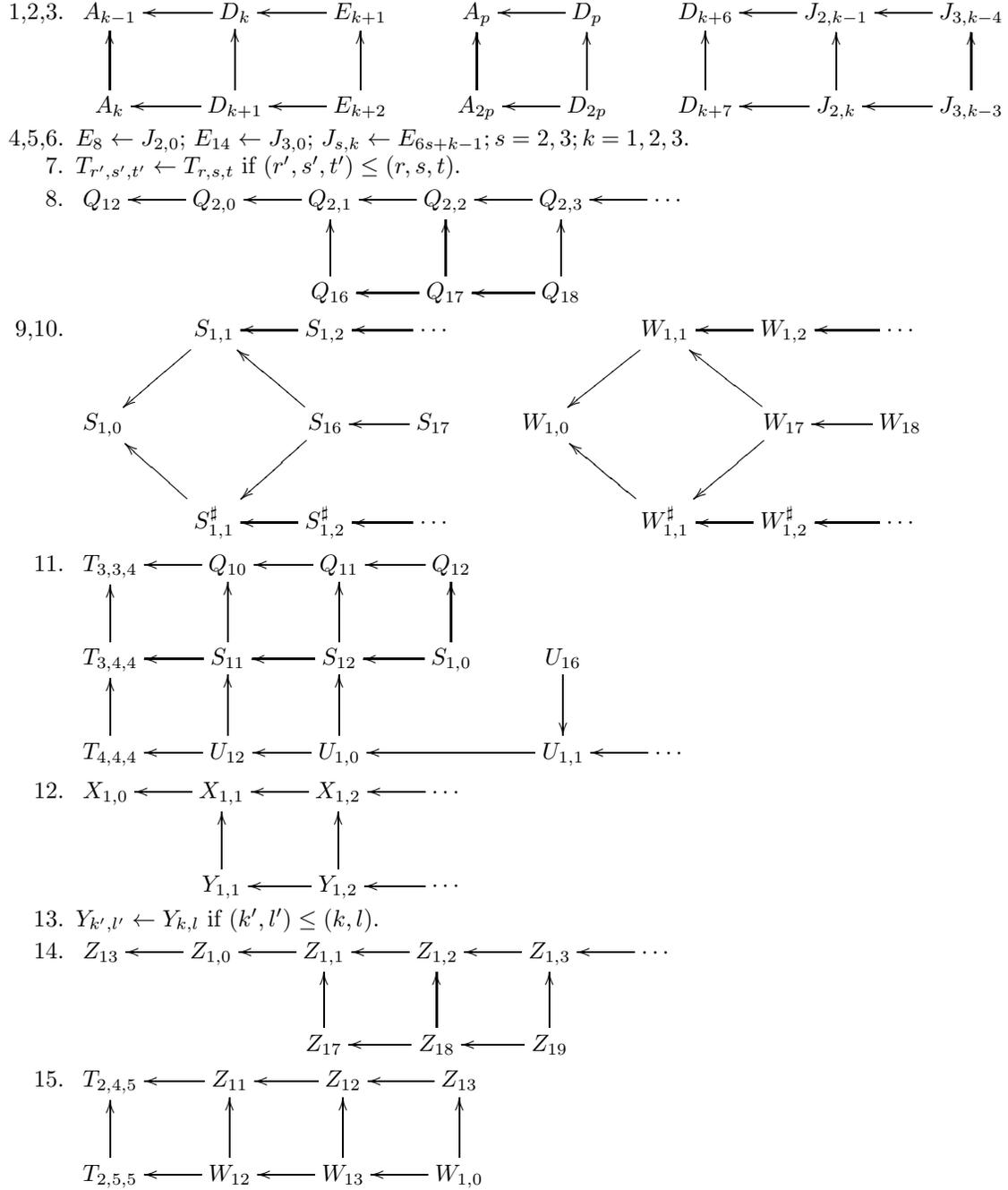
\begin{itemize}
\item[1,2,3.] $\xymatrix{
A_{k-1}\ar@{<-}[r]\ar@{<-}[d]&D_k\ar@{<-}[r]\ar@{<-}[d]&E_{k+1}\ar@{<-}[d]&A_{p}\ar@{<-}[r]\ar@{<-}[d]&D_{p}\ar@{<-}[d]&D_{k+6}\ar@{<-}[r]\ar@{<-}[d]&J_{2,k-1}\ar@{<-}[r]\ar@{<-}[d]&J_{3,k-4}\ar@{<-}[d]\\
A_{k}\ar@{<-}[r]&D_{k+1}\ar@{<-}[r]& E_{k+2}&A_{2p}\ar@{<-}[r]&D_{2p}&D_{k+7}\ar@{<-}[r]&J_{2,k}\ar@{<-}[r]& J_{3,k-3}\\
}$
\item[4,5,6.] $E_{8}\leftarrow J_{2,0} $; $E_{14}\leftarrow J_{3,0} $; $J_{s,k}\leftarrow E_{6s+k-1}; s=2,3; k=1,2,3$.
\item[7.] $T_{r',s',t'}\leftarrow T_{r,s,t} $ if $(r',s',t')\leq (r,s,t)$.
\item[8.] $\xymatrix{
Q_{12}\ar@{<-}[r]&Q_{2,0}\ar@{<-}[r]& Q_{2,1}\ar@{<-}[r]\ar@{<-}[d]&Q_{2,2}\ar@{<-}[r]\ar@{<-}[d]&Q_{2,3}\ar@{<-}[r]\ar@{<-}[d]&\cdots\\
&&Q_{16}\ar@{<-}[r]&Q_{17}\ar@{<-}[r]&Q_{18}&
}$
\item[9,10.] $ \xymatrix{
& S_{1,1}\ar@{<-}[r]\ar@{->}[dl]\ar@{<-}[dr]&S_{1,2}\ar@{<-}[r]&\cdots&& W_{1,1}\ar@{<-}[r]\ar@{->}[dl]\ar@{<-}[dr]&W_{1,2}\ar@{<-}[r]&\cdots\\
S_{1,0}\ar@{<-}[dr]&&S_{16}\ar@{<-}[r]\ar@{->}[dl]&S_{17}&W_{1,0}\ar@{<-}[dr]&&W_{17}\ar@{<-}[r]\ar@{->}[dl]&W_{18}\\
&S^{\sharp}_{1,1}\ar@{<-}[r]&S^{\sharp}_{1,2}\ar@{<-}[r]&\cdots&&W^{\sharp}_{1,1}\ar@{<-}[r]&W^{\sharp}_{1,2}\ar@{<-}[r]&\cdots
}$
\item[11.] $\xymatrix{
T_{3,3,4}\ar@{<-}[r]\ar@{<-}[d]& Q_{10}\ar@{<-}[r]\ar@{<-}[d]&Q_{11}\ar@{<-}[r]\ar@{<-}[d]&Q_{12}\ar@{<-}[d]&&\\
T_{3,4,4}\ar@{<-}[r]\ar@{<-}[d]& S_{11}\ar@{<-}[r]\ar@{<-}[d]&S_{12}\ar@{<-}[r]\ar@{<-}[d]&S_{1,0}&U_{16}\ar@{->}[d]&\\
T_{4,4,4}\ar@{<-}[r]& U_{12}\ar@{<-}[r]&U_{1,0}\ar@{<-}[rr]&&U_{1,1}\ar@{<-}[r]&\cdots
}$
\item[12.] $\xymatrix{
X_{1,0}\ar@{<-}[r]& X_{1,1}\ar@{<-}[r]\ar@{<-}[d]&X_{1,2}\ar@{<-}[r]\ar@{<-}[d]&\cdots\\
&Y_{1,1}\ar@{<-}[r]&Y_{1,2}\ar@{<-}[r]&\cdots
}$
\item[13.] $Y_{k',l'}\leftarrow Y_{k,l}$ if $(k',l')\leq (k,l)$.
\item[14.] $\xymatrix{
Z_{13}\ar@{<-}[r]&Z_{1,0}\ar@{<-}[r]& Z_{1,1}\ar@{<-}[r]\ar@{<-}[d]&Z_{1,2}\ar@{<-}[r]\ar@{<-}[d]&Z_{1,3}\ar@{<-}[r]\ar@{<-}[d]&\cdots\\
&&Z_{17}\ar@{<-}[r]&Z_{18}\ar@{<-}[r]&Z_{19}&
}$

\item[15.] $\xymatrix{
T_{2,4,5}\ar@{<-}[r]\ar@{<-}[d]& Z_{11}\ar@{<-}[r]\ar@{<-}[d]&Z_{12}\ar@{<-}[r]\ar@{<-}[d]&Z_{13}\ar@{<-}[d]\\
T_{2,5,5}\ar@{<-}[r]&W_{12}\ar@{<-}[r]&W_{13}\ar@{<-}[r]&W_{1,0}
}$
\end{itemize}

$$$$
The following corollaries follow from these theorems and concrete calculations. Let $f\in K[[{\bf x}]]$, with $p=\mathrm{char}(K)>0$ such that $\mathrm{rmod}(f)\leq 2$. Then
\begin{corollary}\label{coro30}
If $p\leq 3$, then $f$ is of type $A$ or $D$.
\end{corollary}
\begin{corollary}\label{coro31}
$$\mu(f)\leq 4p.$$
\end{corollary}
\begin{corollary}\label{coro32}
The $\mu$-constant stratum of $f$ is a linear space, and hence smooth.
\end{corollary}
\begin{corollary}\label{coro33}
$$\mathrm{rmod}(f)=\mathrm{pmod}(f).$$
\end{corollary}
\section{Singularity Determinator}\label{sec4}
In \cite{Arn76} Arnol'd supplied lists of normal forms which contain all the singularities with the modality number $\mathrm{mod}=0, 1, 2$, all the singularities with Milnor number $\mu\leq 16$, all the singularities of corank 2 with nonzero 4-jet, all the singularities of corank 3 with a 3-jet, which determine an irreducible cubic, and some other singularities. The proof of Arnol'd is organized as a determinator consisting of 105 theorems. We follow this scheme and organize our proof a singularity determinator of 152 theorems. This gives an algorithm finding for every given sigularity its place in the list of $\S$ \ref{sec3}. 

{\bf Notations:}

$ \begin{array}{ll}
\Rightarrow & \text{``implies"}.\\
\mapsto &\text{``see"}.\\
\mathrm{crk} &\text{the corank of the Hessian of} f \text{at the origin}.\\
\Delta &\text{discriminant, } \Delta=4(a^3+b^3)+27- a^2 b^2 - 18ab.\\
j_{\{ {\bf x}^{\alpha_i}\}}f({\bf x})& \text{quasijet of } f \text{ determined by } \{{\bf x}^{\alpha_i}\}, \text{ defined as follows}. 
\end{array}$

Here $\{\alpha_i\}$ is a system of $n$ points defining an affine hyperplane $H$ in $\mathbb{R}^n$. Let $v\colon \mathbb{R}^n\to \mathbb{R}$ be the linear form defining $H$ with $v(\alpha_i)=1$ for all $i$. Then $j_{\{ {\bf x}^{\alpha_i}\}}f$ is the image of $f$ in $K[[{\bf x}]]$ modulo the ideal generated by $x^{\alpha},\ v(\alpha)>1$. 
\subsection{Singularity determinator in characteristic $\geq 5$}
\begin{itemize}
\item[{\bf 1.}] $\mu(f)<\infty\Rightarrow $ one of the four possibilities holds:

$\begin{array}{lllll}
\mathrm{crk}(f)&\leq & 1 &\mapsto &{\bf 2},\\
&= & 2 &\mapsto &{\bf 4-73},\\
&= & 3 &\mapsto &{\bf 74-119},\\
&> & 3 &\mapsto &{\bf 120}.
\end{array}$
\item[{\bf 2.}] $\mathrm{crk}(f)\leq 1,\mu< 3p \Rightarrow \mathrm{rmod}(f)=\lfloor \mu/p\rfloor$ and $ f\in A_{\mu}$.
\item[{\bf 3.}] $\mathrm{crk}(f)\leq 1,\mu\geq 3p \Rightarrow \mathrm{rmod}(f)\geq 3$.
\end{itemize}

{\bf Corank 2 Singularities}
\begin{itemize}
\item[] Through theorems {\bf 4--73}, $f\in K[[x,y]]$.
\item[{\bf 4.}] $j^2(f)=0\Rightarrow $ one of the four possibilities holds:

$\begin{array}{lllll}
j^3f&\sim_r & x^2y+y^3 &\mapsto &{\bf 5},\\
&\sim_r & x^2 y &\mapsto &{\bf 6},\\
&\sim_r & x^3 &\mapsto &\textbf{9--30},\\
&=& 0 &\mapsto &\textbf{31--73}.
\end{array}$
\item[{\bf 5.}] $j^3(f)=x^2y+y^3 \Rightarrow f\in D_4$.
\item[{\bf 6.}] $j^3(f)=x^2y \Rightarrow f\sim_r x^2y+\alpha(y), j^3(\alpha)=0\mapsto \textbf{7--8}$.
\item[{\bf 7.}] $f=x^2y+\alpha(y), j^3(\alpha)=0, k:=\mu(\alpha)\leq 3p-1\Rightarrow f\in D_{k+2}$.
\item[{\bf 8.}] $f=x^2y+\alpha(y), j^3(\alpha)=0, \mu(\alpha)\geq 3p\Rightarrow \mathrm{rmod}(f)\geq 3$.
\item[] Through theorems {\bf 9--12}, $k=1,2,3$ for $p>7$, $k=1,2$ for $p=7$,$k=1$ for $p=5$.
\item[{\bf 9.}] $j_{x^3,y^{3k} }f(x,y)=x^ 3 \Rightarrow $ one of the four possibilities holds:

$\begin{array}{lllll}
j_{x^3,y^{3k+1} }f(x,y)&\sim_r & x^3+y^{3k+1} &\mapsto &{\bf 10},\\
j_{x^3,xy^{2k+1} }f(x,y)&\sim_r & x^3+xy^{2k+1} &\mapsto &{\bf 11},\\
j_{x^3,y^{3k+2} }f(x,y)&\sim_r & x^3+y^{3k+2} &\mapsto &{\bf 12,13},\\
j_{x^3,y^{3k+2} }f(x,y)&=& x^3 &\mapsto &{\bf 13,26}.
\end{array}$
\item[{\bf 10.}] $j_{x^3,y^{3k+1} }f(x,y)= x^3+y^{3k+1}$ and $3k+1<p$ $\Rightarrow f\in E_{6k}$.
\item[{\bf 11.}] $j_{x^3,xy^{2k+1} }f(x,y)= x^3+xy^{2k+1}$ and $3k+1<p$ $\Rightarrow f\in E_{6k+1}$.
\item[{\bf 12.}] $j_{x^3,y^{3k+2} }f(x,y)= x^3+y^{3k+2}$ and $3k+2<p$ $\Rightarrow f\in E_{6k+2}$.
\item[{\bf 13.}] $p=5$ and $j_{x^3,y^{5} }f(x,y)= x^3+ay^{5}\Rightarrow $ one of the three possibilities holds:

$\begin{array}{lllll}
j_{x^3,y^{6} }f(x,y)&\sim_r & x^3+bx^2y^2+y^{6}+ay^5,\ 4b^3+27\neq 0&\mapsto &{\bf 14},\\
&\sim_r & x^3+x^2y^2 +ay^5&\mapsto &{\bf 15,16},\\
&\sim_r & x^3 +ay^5&\mapsto &{\bf 17}.
\end{array}$
\item[{\bf 14.}] $p=5$ and $j_{x^3,y^{6} }f= x^3+bx^2y^2+y^{6}+ay^5,\ 4b^3+27\neq 0\Rightarrow f\in J_{2,0}$. 
\item[{\bf 15.}] $p=5,j_{x^3,y^{6} }f= x^3+x^2y^2+ay^5$ and $\mu<14\Rightarrow f\in J_{2,q}$ with $q=\mu-10>0$.
\item[{\bf 16.}] $p=5,j_{x^3,y^{6} }f= x^3+x^2y^2+ay^5$ and $\mu\geq 14\Rightarrow \mathrm{rmod}(f)\geq 3$.
\item[{\bf 17.}] $p=5,j_{x^3,y^{6} }f= x^3+ay^5\Rightarrow $ one of the two possibilities holds:

$\begin{array}{lllll}
j_{x^3,y^{7} }f(x,y)&\sim_r & x^3+ay^5+y^{7} &\mapsto &{\bf 18},\\
j_{x^3,y^{7} }f(x,y)&=& x^3+ay^5 &\mapsto &{\bf 19}.
\end{array}$
\item[{\bf 18.}] $p=5,j_{x^3,y^{7} }f= x^3+ay^5+y^7\Rightarrow f\in E_{12}$.
\item[{\bf 19.}] $p=5,j_{x^3,y^{7} }f= x^3+ay^5\Rightarrow \mathrm{rmod}(f)\geq 3$.
\item[{\bf 20.}] $p=7$ and $j_{x^3,y^{7} }f= x^3+ay^7\Rightarrow $ one of the three possibilities holds:

$\begin{array}{lllll}
j_{x^3,xy^{5} }f(x,y)&\sim_r & x^3+xy^5+ay^7&\mapsto &{\bf 21},\\
j_{x^3,y^{8} }f(x,y)&\sim_r & x^3+y^8 +ay^7&\mapsto &{\bf 22},\\
j_{x^3,y^{8} }f(x,y)&= & x^3 +ay^7&\mapsto &{\bf 23}.
\end{array}$
\item[{\bf 21.}] $p=7$ and $j_{x^3,xy^{5} }f= x^3+xy^5+ay^7 \Rightarrow f\in E_{13}$. 
\item[{\bf 22.}] $p=7$ and $j_{x^3,y^{8} }f= x^3+y^8+ay^7 \Rightarrow f\in E_{14}$.
\item[{\bf 23.}] $p=7$ and $j_{x^3,y^{8} }f= x^3+ay^7 \Rightarrow \mathrm{rmod}(f)\geq 3$.  
\item[{\bf 24.}] $p=11$ and $j_{x^3,y^{11} }f= x^3+y^{11} \Rightarrow \mathrm{rmod}(f)\geq 3$.  
\item[{\bf 25.}] $j_{x^3,y^{11} }f(x,y)= x^3\Rightarrow f\in \langle x,y^4\rangle^3 \Rightarrow \mathrm{rmod}(f)\geq 3$.
\item[] Through theorems {\bf 26--29}, $k=2,3$.
\item[{\bf 26.}] $j_{x^3,y^{3k-1} }f(x,y)=x^ 3 \Rightarrow $ one of the three possibilities holds:

$\begin{array}{lllll}
j_{x^3,y^{3k} }f(x,y)&\sim_r & x^3+ax^2y^k+y^{3k},\ 4a^3+27\neq 0&\mapsto &{\bf 27},\\
&\sim_r & x^3+x^2y^k &\mapsto &{\bf 28,29},\\
&\sim_r & x^3 &\mapsto &{\bf 9,30}.
\end{array}$
\item[{\bf 27.}] $j_{x^3,y^{3k} }f(x,y)= x^3+ax^2y^k+y^{3k},\ 4a^3+27\neq 0$ and $3k<p$ $\Rightarrow f\in J_{k,0}$
\item[{\bf 28.}] $j_{x^3,y^{3k} }f(x,y)= x^3+x^2y^k, 3k<p$ and $\mu-3k+2<2p\Rightarrow f\in J_{k,q}$ with $q=\mu-6k+2$. 
\item[{\bf 29.}] $j_{x^3,y^{3k} }f(x,y)= x^3+x^2y^k, 3k<p$ and $\mu-3k+2\geq 2p\Rightarrow \mathrm{rmod}(f)\geq 3$.
\item[{\bf 30.}] $p=7$ and $j_{x^3,y^{6} }f(x,y)= x^3+x^2y^2$ and $\mu<18\Rightarrow f\in J_{2,q}$ with $q=\mu-10>0$. 
\end{itemize}

{\bf Series X}
\begin{itemize}
\item[{\bf 31.}] $j^3f=0\Rightarrow $ one of the six possibilities holds:

$\begin{array}{lllll}
j^4f&\sim_r & x^4+ax^2y^2+y^4,\ a^2+4\neq 0 &\mapsto &{\bf 32},\\
&\sim_r & x^4+x^2y^2 &\mapsto &{\bf 33,34},\\
&\sim_r & x^2y^2 &\mapsto &\textbf{35--38},\\
&\sim_r & x^3y &\mapsto &\textbf{39},\\
&\sim_r & x^4 &\mapsto &\textbf{54},\\
&=& 0 &\mapsto &\textbf{71}.
\end{array}$
\item[{\bf 32.}] $j^4(f)=x^4+ax^2y^2+y^4,\ a^2+4\neq 0 \Rightarrow f\in X_9$.
\item[{\bf 33.}] $j^4(f)=x^4+x^2y^2$ and $\mu(f)<2p+5$ $\Rightarrow f\in X_{1,q}$.
\item[{\bf 34.}] $j^4(f)=x^4+x^2y^2$ and $\mu(f)\geq 2p+5$ $\Rightarrow \mathrm{rmod}(f)\geq 3$.
\item[{\bf 35.}] $j^4(f)=x^2y^2\Rightarrow f=f_1\cdot f_2$ with $\mathrm{mt}(f_1)=\mathrm{mt}(f_2)=2$ and $2\leq \mu(f_1)\leq \mu(f_2)\Rightarrow \textbf{36}$.
\item[] Through theorems {\bf 36--39}, $1\leq r:=\mu(f_1)-1\leq s:=\mu(f_2)-1$.
\item[{\bf 36.}] $\mu(f_1)\geq p$ or $\mu(f_2)\geq 2p \Rightarrow \mathrm{rmod}(f)\geq 3$.
\item[{\bf 37.}] $\mu(f_2)< p \Rightarrow \mathrm{rmod}(f)=1 \text{ and } f\in Y_{r,s}$.
\item[{\bf 38.}] $\mu(f_1)< p$ and $p\leq \mu(f_2)< 2p \Rightarrow \mathrm{rmod}(f)=2 \text{ and } f\in Y_{r,s}$.
\item[{\bf 39.}] $j^4(f)=x^3y\Rightarrow j_{x^3y,y^{4}}f =x^3y\mapsto  \textbf{40,44}$.
\end{itemize}

{\bf Series Z}
\begin{itemize}
\item[] Through theorems {\bf 40--43}, $5<p$ and $q=1,2$.
\item[{\bf 40.}] $j_{x^3y,y^{3q+1}}f =x^3y\Rightarrow $ one of the four possibilities holds:

$\begin{array}{lllll}
j_{x^3y,y^{3q+2}}f&\sim_r & x^3y+y^{3q+2} &\mapsto &{\bf 41},\\
j_{x^3y,xy^{2q+2}}f&\sim_r & x^3y+xy^{2q+2} &\mapsto &{\bf 42},\\
j_{x^3y,y^{3q+3}}f&\sim_r & x^3y+y^{3q+3} &\mapsto &\textbf{43},\\
j_{x^3y,y^{3q+3}}f&= & x^3y &\mapsto &\textbf{49--53}.
\end{array}$
\item[{\bf 41.}] $j_{x^3y,y^{3q+2}}f= x^3y+y^{3q+2}, 3q+2<p\Rightarrow   f\in Z_{6q+5}$.
\item[{\bf 42.}] $j_{x^3y,xy^{2q+2}}f= x^3y+xy^{2q+2}, 3q+3<p\Rightarrow   f\in Z_{6q+6}$.
\item[{\bf 43.}] $j_{x^3y,y^{3q+3}}f= x^3y+y^{3q+3}, 3q+3<p\Rightarrow   f\in Z_{6q+7}$.
\item[{\bf 44.}] $p=5$ and $j_{x^3y,y^{4}}f =x^3y\Rightarrow $ one of the three possibilities holds:

$\begin{array}{lllll}
j_{x^3y,xy^{4}}f&\sim_r & x^3y+xy^{4}+ay^{5} &\mapsto &{\bf 45},\\
j_{x^3y,y^{6}}f&\sim_r & x^3y+y^{6}+ay^{5} &\mapsto &\textbf{46},\\
j_{x^3y,y^{6}}f&= & x^3y+ay^{5} &\mapsto &\textbf{47}.
\end{array}$

\item[{\bf 45.}] $p=5$ and $j_{x^3y,xy^{4}}f= x^3y+xy^{4}+ay^{5}\Rightarrow f\in Z_{12}$.
\item[{\bf 46.}] $p=5$ and $j_{x^3y,xy^{4}}f= x^3y+y^{6}+ay^{5}\Rightarrow f\in Z_{13}$.
\item[{\bf 47.}] $p=5$ and $j_{x^3y,xy^{4}}f= x^3y+ay^{5}\Rightarrow \mathrm{rmod}(f)\geq 3$.
\item[{\bf 48.}] $p=7$ and $j_{x^3y,y^{7}}f =x^3y\Rightarrow \mathrm{rmod}(f)\geq 3$. 
\item[{\bf 49.}] $j_{x^3y,y^{6}}f =x^3y\Rightarrow $ one of the three possibilities holds:

$\begin{array}{lllll}
j_{x^3y,y^{7}}f&= & y(x^3+bx^2y^2+y^{9}),\ 4b^3+27\neq 0 &\mapsto &{\bf 50}\\
&= &y(x^3+x^2y^{2}) &\mapsto &{\bf 51,52}\\
&= & x^3y &\mapsto &\textbf{53}.
\end{array}$
\item[{\bf 50.}] $j_{x^3y,y^{7}}f= y(x^3+bx^2y^2+y^{9}),\ 4b^3+27\neq 0\Rightarrow   f\in Z_{1,0}$.
\item[{\bf 51.}] $j_{x^3y,y^{7}}f= y(x^3+x^2y^{2})$ and $\mu-8<p\Rightarrow   f\in Z_{1,r}$ with $r=\mu-15>0$.
\item[{\bf 52.}]$j_{x^3y,y^{7}}f= y(x^3+x^2y^{2})$ and $\mu-8\geq p\Rightarrow \mathrm{rmod}(f)\geq 3$.
\item[{\bf 53.}] $j_{x^3y,y^{9}}f= x^3y \Rightarrow   f\in \langle y\rangle\cdot \langle x,y^3\rangle^3\Rightarrow \mathrm{rmod}(f)\geq 3$.
\end{itemize}

{\bf Series W}
\begin{itemize}
\item[] Through theorems {\bf 54--56}, $5< p$.
\item[{\bf 54.}] $j^4f =x^4\Rightarrow j_{x^4,y^4}f=x^4\Rightarrow $ one of the three possibilities holds:

$\begin{array}{lllll}
j_{x^4,y^{5}}f&\sim_r & x^4+y^{5} &\mapsto &{\bf 55},\\
j_{x^4,xy^{4}}f&\sim_r & x^4+xy^{4} &\mapsto &{\bf 56},\\
j_{x^4,xy^{4}}f&= & x^4 &\mapsto &\textbf{57,60}.
\end{array}$
\item[{\bf 55.}] $j_{x^4,y^{5}}f=x^4+y^{5}\Rightarrow   f\in W_{12}$.
\item[{\bf 56.}] $j_{x^4,xy^{4}}f= x^4+xy^{4}\Rightarrow   f\in W_{13}$.
\item[{\bf 57.}] $p=5$ and $j_{x^4,y^{4}}f=x^4\Rightarrow $
one of the two possibilities holds:

$\begin{array}{lllll}
j_{x^4,xy^{4}}f&\sim_r & x^4+xy^{4}+ay^5 &\mapsto &{\bf 58},\\
j_{x^4,xy^{4}}f&= & x^4+ay^5 &\mapsto &\textbf{59}.
\end{array}$
\item[{\bf 58.}] $p=5$ and $j_{x^4,xy^{4}}f=x^4+xy^{4}+ay^5\Rightarrow   f\in W_{13}$.
\item[{\bf 59.}] $p=5$ and $j_{x^4,xy^{4}}f= x^4+ay^5\Rightarrow   \mathrm{rmod}(f)\geq 3$.
\item[] Through theorems {\bf 60--70}, $5< p$.
\item[{\bf 60.}] $j_{x^4,xy^{4}}f=x^4\Rightarrow $ one of the three possibilities holds:

$\begin{array}{lllll}
j_{x^4,y^{6}}f&\sim_r & x^4+bx^2y^{3}+y^{6},\ b^2\neq 4&\mapsto &{\bf 61},\\
&\sim_r & x^4+x^2y^{3} &\mapsto &{\bf 62,63},\\
&\sim_r & (x^2+y^{3})^2 &\mapsto &{\bf 64,65}\\
&= & x^4 &\mapsto &\textbf{66}.
\end{array}$
\item[{\bf 61.}] $j_{x^4,y^{6}}f= x^4+bx^2y^{3}+y^{6},\ b^2\neq 4\Rightarrow   f\in W_{1,0}$.
\item[{\bf 62.}] $j_{x^4,y^{6}}f= x^4+x^2y^{3}$ and $\mu-8<p\Rightarrow   f\in W_{1,q} (q=\mu-15> 0)$.
\item[{\bf 63.}] $j_{x^4,y^{6}}f= x^4+x^2y^{3}$ and $\mu-8\geq p\Rightarrow\mathrm{rmod}(f)\geq 3$.
\item[{\bf 64.}] $j_{x^4,y^{6}}f= (x^2+y^{3})^2$ and $\mu-8<p\Rightarrow   f\in W^{\sharp}_{k,q} (q=\mu-15> 0)$.
\item[{\bf 65.}] $j_{x^4,y^{6}}f= (x^2+y^{3})^2$ and $\mu-8\geq p\Rightarrow\mathrm{rmod}(f)\geq 3$.
\item[{\bf 66.}] $j_{x^4,y^{6}}f=x^4\Rightarrow $ one of the three possibilities holds:

$\begin{array}{lllll}
j_{x^4,xy^{5}}f&\sim_r & x^4+xy^{5}&\mapsto &{\bf 67},\\
j_{x^4,y^{7}}f&\sim_r & x^4+y^{7} &\mapsto &{\bf 68,69},\\
&= & x^4 &\mapsto &\textbf{70}.
\end{array}$
\item[{\bf 67.}] $j_{x^4,xy^{5}}f= x^4+xy^{5}\Rightarrow   f\in W_{17}$.
\item[{\bf 68.}] $j_{x^4,y^{7}}f= x^4+y^{7}$ and $p>7\Rightarrow   f\in W_{18}$.
\item[{\bf 69.}] $p=7$ and $j_{x^4,y^{7}}f= x^4+y^{7}\Rightarrow \mathrm{rmod}(f)\geq 3$.
\item[{\bf 70.}] $j_{x^4,y^{7}}f=x^4\Rightarrow \mathrm{rmod}(f)\geq 3$.
\item[] Through theorems {\bf 53--55}, $5< p$.
\item[{\bf 71.}] $j^{4}f=0\Rightarrow $ one of the two possibilities holds:

$\begin{array}{lllll}
j_{5}f&\sim_r & x^4y+ax^3y^2+bx^2y^3+xy^{4},\ \Delta\neq 0,\ ab\neq 9&\mapsto &{\bf 54},\\
j_{5}f&\text{ is } & \text{degenerate} &\mapsto &{\bf 55}.
\end{array}$
\item[{\bf 72.}] $j_{5}f= x^4y+ax^3y^2+bx^2y^3+xy^{4},\ \Delta\neq 0,\ ab\neq 9$ $\Rightarrow$  $f\sim_r  x^4y+ax^3y^2+bx^2y^3+xy^{4}+cx^3y^3$ with $\Delta\neq 0,\ ab\neq 9$ and therefore $\mathrm{rmod}(f)\geq 3$.
\item[{\bf 73.}] If $j_{5}f$ is degenerate $\Rightarrow   \mathrm{rmod}(f)\geq 3$.
\end{itemize}

{\bf Corank 3 Singularities}
\begin{itemize}
\item[] Through theorems {\bf 74--120}, $f\in K[[x,y,z]]$.
\item[{\bf 74.}] $j^2f(x,y,z)=0\Rightarrow $ one of the ten possibilities holds:

$\begin{array}{lllll}
j^3f&\sim_r & x^3+y^3+z^3+axyz, a^3 +27\neq 0&\mapsto &{\bf 75},\\
&\sim_r & x^3 + y^3 + xyz &\mapsto &{\bf 76},\\
&\sim_r & x^3+xyz &\mapsto &{\bf 79},\\
&\sim_r & xyz &\mapsto &{\bf 80},\\
&\sim_r & x^3+yz^2 &\mapsto &{\bf 81-91},\\
&\sim_r & x^2z+yz^2 &\mapsto &{\bf 92-106},\\
&\sim_r & x^3+xz^2 &\mapsto &{\bf 107-117},\\
&\sim_r & x^2y &\mapsto &{\bf 118},\\
&\sim_r & x^3 &\mapsto &{\bf 119},\\
&= & 0 &\mapsto &{\bf 120}.
\end{array}$
\end{itemize}

{\bf Series T}
\begin{itemize}
\item[{\bf 75.}] $j^3f(x,y,z)=x^3+y^3+z^3+axyz, a^3 +27\neq 0\Rightarrow f\in P_8$.
\item[{\bf 76.}] $j^3f(x,y,z)=x^3 + y^3 + xyz\Rightarrow f\sim_r x^3 + y^3 + xyz+\alpha(z),\ j^3\alpha=0 \mapsto {\bf 77,78}$.
\item[{\bf 77.}] $f=x^3 + y^3 + xyz+\alpha(z),\ j^3\alpha=0, q:=\mu(\alpha)+1<2p\Rightarrow f\in P_{q+5}=T_{3,3,q}(q>3)$.
\item[{\bf 78.}] $f=x^3 + y^3 + xyz+\alpha(z),\ j^3\alpha=0, \mu(\alpha)+1\geq 2p\Rightarrow \mathrm{rmod}(f)\geq 3$.
\item[{\bf 79.}] $j^3f(x,y,z)= x^3+xyz \Rightarrow f=x^3+xyz+\alpha(y)+\beta(z),\ j^3(\alpha,\beta)=0$ and $q:=\mu(\alpha)+1\leq r:=\mu(\beta)+1$ $\Rightarrow$ one of the three possibilities holds:

$\begin{array}{llll}
(i)&r<p&\Rightarrow &\mathrm{rmod}(f)=1 \text{ and } f\in T_{3,q,r},\\
(ii)&q<p\leq r<2p&\Rightarrow& \mathrm{rmod}(f)=2 \text{ and } f\in T_{3,q,r},\\
(iii) &\text{ otherwise }&\Rightarrow&\mathrm{rmod}(f)\geq 3.
\end{array}$
\item[{\bf 80.}] $j^3f(x,y,z)= xyz\Rightarrow f\sim_r xyz+\alpha(x)+\beta(y)+\gamma(z),\ j^3(\alpha,\beta,\gamma)=0$ and

 $q:=\mu(\alpha)+1\leq r:=\mu(\beta)+1\leq s:=\mu(\gamma)+1$ $\Rightarrow$ one of the three possibilities holds:

$\begin{array}{llll}
(i)&s<p&\Rightarrow &\mathrm{rmod}(f)=1 \text{ and } f\in T_{q,r,s},\\
(ii)&r<p\leq s<2p&\Rightarrow& \mathrm{rmod}(f)=2 \text{ and } f\in T_{q,r,s},\\
(iii) &\text{ otherwise }&\Rightarrow&\mathrm{rmod}(f)\geq 3.
\end{array}$.
\end{itemize}

{\bf Series Q}
\begin{itemize}
\item[] Through theorems {\bf 81--91}, $\varphi=x^3+yz^2, j^{*}_{\lambda}=j_{yz^2,x^3,\lambda},$ ($\lambda$ is a polynomial).
\item[{\bf 81.}] $j^3f=\varphi\Rightarrow f\sim_r \varphi+\alpha(y)+x\beta(y), j^{3}(\alpha,x\beta)=0 \mapsto {\bf 82}$.
\item[] Through theorems {\bf 82--85}, $k=1,2$.
\item[{\bf 82.}] $f=\varphi+\alpha(y)+x\beta(y), j^{*}_{y^{3k}}f=\varphi\Rightarrow $ one of the four possibilities holds:

$\begin{array}{lllll}
j^{*}_{y^{3k+1}} f&\sim_r & \varphi+y^{3k+1}&\mapsto &{\bf 83},\\
j^{*}_{xy^{2k+1}} f&\sim_r & \varphi+xy^{2k+1} &\mapsto &{\bf 84},\\
j^{*}_{y^{3k+2}} f&\sim_r & \varphi+y^{3k+2} &\mapsto &{\bf 85,86},\\
j^{*}_{y^{3k+2}} f&\sim_r & \varphi &\mapsto &{\bf 87}.
\end{array}$
\item[{\bf 83.}] $j^{*}_{y^{3k+1}} f= \varphi+y^{3k+1}$ and $3k+1<p\Rightarrow f\in Q_{6k+4}$.
\item[{\bf 84.}] $j^{*}_{xy^{2k+1}} f= \varphi+xy^{2k+1}$ and $3k+1<p\Rightarrow f\in Q_{6k+5}$.
\item[{\bf 85.}] $j^{*}_{y^{3k+2}} f=  \varphi+y^{3k+2}$ and $3k+2<p\Rightarrow f\in Q_{6k+6}$.
\item[{\bf 86.}] $p=5$ and  $j^{*}_{xy^{3}} f=  \varphi\Rightarrow \mathrm{rmod}(f)\geq 3$.
\item[{\bf 87.}] $f=\varphi+\alpha(y)+x\beta(y), j^{*}_{y^{5}}f=\varphi\Rightarrow $ one of the three possibilities holds:

$\begin{array}{lllll}
j^{*}_{y^{6}} f&\sim_r & \varphi+ax^2y^{2}+xy^{4},a^2\neq 4&\mapsto &{\bf 88},\\
&\sim_r & \varphi+x^2y^{2} &\mapsto &{\bf 89,90},\\
&= & \varphi&\mapsto &{\bf 91}.
\end{array}$
\item[{\bf 88.}] $j^{*}_{y^{6}} f= \varphi+ax^2y^{2}+xy^{4},a^2\neq 4\Rightarrow f\in Q_{2,0}$.
\item[{\bf 89.}] $j^{*}_{y^{6}} f= \varphi+x^2y^{2}$ and $\mu-5<p\Rightarrow f\in Q_{2,q}(q=\mu-12> 0)$.
\item[{\bf 90.}] $j^{*}_{y^{6}} f= \varphi+x^2y^{2}$ and $\mu-5\geq p\Rightarrow \mathrm{rmod}(f)\geq 3$.
\item[{\bf 91.}] $p=7$ and $j^{*}_{y^{6}}f=\varphi\Rightarrow \mathrm{rmod}(f)\geq 3$. 

\end{itemize}

{\bf Series S}
\begin{itemize}
\item[] Through theorems {\bf 92--106}, $\varphi=x^2z+yz^2, j^{*}_{\lambda}=j_{x^2y,yz^2,\lambda},$ ($\lambda$ is a polynomial).
\item[{\bf 92.}] $j^3f=\varphi \Rightarrow f=\varphi+\alpha(y)+x\beta(y)+z\gamma(y), j^{3}(\alpha,x\beta,z\gamma)=0\mapsto {\bf 93}$. 
\item[{\bf 93.}] $f=\varphi+\alpha(y)+x\beta(y)+z\gamma(y), j^{*}_{y^{3}} f = \varphi\Rightarrow $ one of the three possibilities holds: 

$\begin{array}{lllll}
j^{*}_{y^{4}} f&\sim_r& \varphi+y^{4}&\mapsto &{\bf 94},\\
j^{*}_{xy^{3}} f&\sim_r & \varphi+xy^{3} &\mapsto &{\bf 95},\\
j^{*}_{xy^{3}} f&= & \varphi &\mapsto &{\bf 96,106}.
\end{array}$
\item[{\bf 94.}] $j^{*}_{y^{4}} f= \varphi+y^{4}\Rightarrow f\in S_{11}$.
\item[{\bf 95.}] $j^{*}_{xy^{3}} f= \varphi+xy^{3}\Rightarrow f\in S_{12}$.
\item[] Through theorems {\bf 96--105}, $p>5$.
\item[{\bf 96.}] $f=\varphi+\alpha(y)+x\beta(y)+z\gamma(y), j^{*}_{xy^{3}} f = \varphi\Rightarrow $ one of the four possibilities holds: 

$\begin{array}{lllll}
j^{*}_{y^{5}} f&\sim_r& \varphi+y^{5}+bzy^{3}, b^2\neq 4&\mapsto &{\bf 97},\\
&\sim_r & \varphi+x^2y^{2} &\mapsto &{\bf 98,99},\\
&\sim_r & \varphi+zy^{3} &\mapsto &{\bf 100,101},\\
&= & \varphi &\mapsto &{\bf 102}.
\end{array}$
\item[{\bf 97.}] $j^{*}_{y^{5}} f= \varphi+y^{5}+bzy^{3}, b^2\neq 4\Rightarrow f\in S_{1,0}$.
\item[{\bf 98.}] $j^{*}_{y^{5}} f=\varphi+x^2y^{2}$ and $\mu-9<p\Rightarrow f\in S_{1,q}(q:=\mu-14>0)$.
\item[{\bf 99.}] $j^{*}_{y^{5}} f=\varphi+x^2y^{2}$ and $\mu-9\geq p\Rightarrow \mathrm{rmod}(f)\geq 3$.
\item[{\bf 100.}] $j^{*}_{y^{5}} f=\varphi+zy^{3}$ and $\mu-9<p\Rightarrow f\in S^{\sharp}_{1,q}(q:=\mu-14>0)$.
\item[{\bf 101.}] $j^{*}_{y^{5}} f=\varphi+zy^{3}$ and $\mu-9\geq p\Rightarrow \mathrm{rmod}(f)\geq 3$.
\item[{\bf 102.}] $f=\varphi+\alpha(y)+x\beta(y)+z\gamma(y), j^{*}_{y^{5}} f = \varphi\Rightarrow $ one of the three possibilities holds: 

$\begin{array}{lllll}
j^{*}_{xy^{4}} f&\sim_r& \varphi+xy^{4}&\mapsto &{\bf 103},\\
j^{*}_{y^{6}} f&\sim_r& \varphi+y^{6}&\mapsto &{\bf 104},\\
j^{*}_{y^{6}} f&= & \varphi &\mapsto &{\bf 105}.
\end{array}$
\item[{\bf 103.}] $j^{*}_{xy^{4}} f= \varphi+xy^{4}\Rightarrow f\in S_{16}$.
\item[{\bf 104.}] $j^{*}_{y^{6}} f= \varphi+y^{6}\Rightarrow f\in S_{17}$.
\item[{\bf 105.}] $j^{*}_{y^{6}} f=  \varphi\Rightarrow \mathrm{rmod}(f)\geq 3$.
\item[{\bf 106.}] $f=\varphi+\alpha(y)+x\beta(y)+z\gamma(y),\ j^{*}_{xy^{3}} f= \varphi$ and $p=5\Rightarrow \mathrm{rmod}(f)\geq 3$.
\end{itemize}

{\bf Series U}
\begin{itemize}
\item[] Through theorems {\bf 107--117}, $\varphi=x^3+xz^2, j^{*}_{\lambda}=j_{x^3,z^3,\lambda},$ ($\lambda$ is a polynomial).
\item[{\bf 107.}] $j^3f=\varphi \Rightarrow f\sim_r \varphi+\alpha(y)+x\beta(y)+z\gamma(y)+x^2\delta(y), j^3(\alpha,x\beta,z\gamma,x^2\delta)=0 \mapsto {\bf 108}$.
\item[{\bf 108.}] $f= \varphi+\alpha(y)+x\beta(y)+z\gamma(y)+x^2\delta(y), j^{*}_{y^{3k}}f=\varphi\Rightarrow $ one of the two possibilities holds: 

$\begin{array}{lllll}
j^{*}_{y^{4}}f&\sim_r& \varphi+y^{4}&\mapsto &{\bf 109},\\
&= & \varphi &\mapsto &{\bf 110}.
\end{array}$
\item[{\bf 109.}] $j^{*}_{y^{4}} f= \varphi+y^{4}\Rightarrow f\in U_{12}$.
\item[{\bf 110.}] $f= \varphi+\alpha(y)+x\beta(y)+z\gamma(y)+x^2\delta(y), j^{*}_{y^{4}}f=\varphi\Rightarrow $ one of the three possibilities holds: 

$\begin{array}{lllll}
j^{*}_{xy^{3}}f&\sim_r& \varphi+xy^{3}+czy^{3}, c(c^2+1)\neq 0&\mapsto &{\bf 111},\\
&\sim_r& \varphi+xy^{3} &\mapsto &{\bf 112,113},\\
&= & \varphi &\mapsto &{\bf 114}.
\end{array}$
\item[{\bf 111.}] $j^{*}_{xy^{3}} f= \varphi+xy^{3}+czy^{3}, c(c^2+1)\neq 0\Rightarrow f\in U_{1,0}$.
\item[{\bf 112.}] $j^{*}_{xy^{3}} f= \varphi+xy^{3}$ and $\mu-13<p\Rightarrow f\in U_{1,q}(q:=\mu-14\geq 0)$.
\item[{\bf 113.}] $j^{*}_{xy^{3}} f= \varphi+xy^{3}$ and $\mu-13\geq p\Rightarrow \mathrm{rmod}(f)\geq 3$.
\item[{\bf 114.}] $f= \varphi+\alpha(y)+x\beta(y)+z\gamma(y)+x^2\delta(y), j^{*}_{xy^{3}}f=\varphi\Rightarrow $ one of the two possibilities holds: 

$\begin{array}{lllll}
j^{*}_{y^{5}}f&\sim_r& \varphi+y^{5}&\mapsto &{\bf 115,116},\\
&= & \varphi &\mapsto &{\bf 116,117}.
\end{array}$
\item[{\bf 115.}] $j^{*}_{y^{5}} f= \varphi+y^{5}$ and $p>5\Rightarrow f\in U_{16}$.
\item[{\bf 116.}] $p=5, f= \varphi+\alpha(y)+x\beta(y)+z\gamma(y)+x^2\delta(y)$ and $j^{*}_{xy^{3}}f=\varphi\Rightarrow \mathrm{rmod}(f)\geq 3$.
\item[{\bf 117.}] $j^{*}_{y^{5}} f= \varphi\Rightarrow \mathrm{rmod}(f)\geq 3$.
\item[{\bf 118.}] $j^3f=x^2y\Rightarrow f\sim_r x^2y+\alpha(y,z)+x\beta(z)$ and then $\mathrm{rmod}(f)\geq 3$.
\item[{\bf 119.}] $j^3f=x^3\Rightarrow \mathrm{rmod}(f)\geq 4$.
\item[{\bf 120.}] $j^3f=0\Rightarrow \mathrm{rmod}(f)\geq 6$.
\end{itemize}
{\bf Corank $> 3$ Singularities}
\begin{itemize}
\item[{\bf 121.}] $\mathrm{crk(f)}>3\Rightarrow \mathrm{rmod}(f)\geq 4$.
\end{itemize}
\subsection{Singularity determinator in characteristic 2}
\begin{itemize}
\item[{\bf 122.}] $\mu(f)<\infty\Rightarrow $ one of the four possibilities holds:

$\begin{array}{lllll}
\mathrm{crk}(f)&\leq & 1 &\mapsto &{\bf 123},\\
&= & 2 &\mapsto &{\bf 124},\\
&\geq & 3 &\mapsto& {\bf 136}.
\end{array}$
\item[{\bf 123.}] $\mathrm{crk}(f)\leq 1 \Rightarrow f\in A_k\ (1\leq k\leq 5)$.
\item[] Through theorems {\bf }, $f\in K[[x,y]]$.
\item[{\bf 124.}] $\mathrm{crk}(f)=2 \Rightarrow $ one of the four possibilities holds:

$\begin{array}{lllll}
j^3f&\sim_r & ax^2+by^2+x^3+y^3 &\mapsto &{\bf 125},\\
&\sim_r & ax^2+by^2+x^2y &\mapsto &{\bf 126},\\
&\sim_r & ax^2+by^2+x^3 &\mapsto &{\bf 131},\\
&= & ax^2+by^2 &\mapsto &{\bf 134,135}.
\end{array}$
\item[{\bf 125.}] $j^3f=ax^2+by^2+x^3+y^3 \Rightarrow f\in D_4$.
\item[{\bf 126.}] $j^3f=ax^2+by^2+x^2y \Rightarrow$ one of the two possibilities holds:

$\begin{array}{lllll}
j^4f&\sim_r & ax^2+by^2+x^2y+xy^3 &\mapsto &{\bf 127},\\
&=& ax^2+by^2+x^2y &\mapsto &{\bf 128,130}.\\
\end{array}$
\item[{\bf 127.}] $j^4f= ax^2+by^2+x^2y+xy^3\Rightarrow f\in D_6$.
\item[{\bf 128.}] $j^4f=ax^2+by^2+x^2y \Rightarrow$ one of the two possibilities holds:

$\begin{array}{lllll}
j^5f&\sim_r & ax^2+by^2+x^2y+xy^4 &\mapsto &{\bf 129},\\
&=& ax^2+by^2+x^2y &\mapsto &{\bf 130}.\\
\end{array}$
\item[{\bf 129.}] $j^5f=ax^2+by^2+x^2y+xy^4 \Rightarrow$ one of the two possibilities holds:

$\begin{array}{lllll}
j^5f&\sim_r & ax^2+by^2+x^2y+xy^4+cxy^5&\mapsto &{\bf 130},\\
&=& ax^2+by^2+x^2y &\mapsto &{\bf 130}.\\
\end{array}$
\item[{\bf 130.}] $j^4f= ax^2+by^2+x^2y\Rightarrow \mathrm{rmod}(f)\geq 3,\mu(f)\geq 8$.
\item[{\bf 131.}] $j^3f=ax^2+by^2+x^3 \Rightarrow$ one of the two possibilities holds:

$\begin{array}{lllll}
j^4f&\sim_r & ax^2+by^2+x^3+xy^3 &\mapsto &{\bf 132},\\
&=& ax^2+by^2+x^3 &\mapsto &{\bf 133}.\\
\end{array}$
\item[{\bf 132.}] $j^4f= ax^2+by^2+x^3+xy^3 \Rightarrow f\in E_7$.
\item[{\bf 133.}] $j^4f= ax^2+by^2+x^3 \Rightarrow \mathrm{rmod}(f)\geq 3$.
\item[{\bf 134.}] $j^3f=ax^2+by^2, (a,b)\neq (0,0) \Rightarrow$ one of the two possibilities holds:

$\begin{array}{lllll}
j^4f&\sim_r & ax^2+by^2+x^3y &\mapsto &{\bf 135},\\
&=& ax^2+by^2 &\mapsto &{\bf 135}.\\
\end{array}$
\item[{\bf 135.}] $j^3f= ax^2+by^2\Rightarrow \mathrm{rmod}(f)\geq 4,\mu(f)\geq 10$.
\item[{\bf 136.}] $\mathrm{crk}(f)\geq 3\Rightarrow \mathrm{rmod}(f)\geq 4,\mu(f)\geq 8$.
\end{itemize}
\subsection{Singularity determinator in characteristic 3}
\begin{itemize}
\item[{\bf 137.}] $\mu(f)<\infty\Rightarrow $ one of the four possibilities holds:

$\begin{array}{lllll}
\mathrm{crk}(f)&\leq & 1 &\mapsto &{\bf 138},\\
&= & 2 &\mapsto &\textbf{139--146},\\
&= & 3 &\mapsto &\textbf{147--151},\\
&> & 3 &\mapsto &\textbf{152}.
\end{array}$
\item[{\bf 138.}] $\mathrm{crk}(f)\leq 1 \Rightarrow f\in A_k\ (1\leq k\leq 8)$.
\end{itemize}
{\bf Corank 2 Singularities}
\begin{itemize}
\item[] Through theorems {\bf 139--146}, $f\in K[[x,y]]$.
\item[{\bf 139.}] $j^2(f)=0\Rightarrow $ one of the four possibilities holds:

$\begin{array}{lllll}
j^3f&\sim_r & x^2y+\epsilon y^3,\epsilon\in \{0,1\} &\mapsto &{\bf 140},\\
&\sim_r & x^3 &\mapsto &\textbf{145},\\
&=& 0 &\mapsto &{\bf 146}.
\end{array}$
\item[{\bf 140.}] $j^3(f)=x^2y+\epsilon y^3,\epsilon\in \{0,1\} \Rightarrow f\sim_r x^2y+g(y),j^2g=0\mapsto {\bf 141}$.
\item[{\bf 141.}] $j^4f= x^2y+g(y),j^2g=0 \Rightarrow$ one of the three possibilities holds:

$\begin{array}{lllllll}
2&<& \mu(g)&<  & 5&\mapsto &{\bf 142},\\
5&<&\mu(g)& < &8&\mapsto &{\bf 143},\\
8&<&\mu(g)&  &&\mapsto &{\bf 144}.
\end{array}$
\item[{\bf 142.}] $2<\mu(g)<5 \Rightarrow f\in D_{5},D_{6}$.
\item[{\bf 143.}] $5<\mu(g)<8 \Rightarrow f\in D_{8},D_{9}$.
\item[{\bf 144.}] $8<\mu(g) \Rightarrow \mathrm{rmod}(f)\geq 3,\mu(f)\geq 11$.
\item[{\bf 145.}] $j^3(f)=x^3 \Rightarrow \mathrm{rmod}(f)\geq 3,\mu(f)\geq 9$.
\item[{\bf 146.}] $j^3(f)=0 \Rightarrow \mathrm{rmod}(f)\geq 3,\mu(f)\geq 9$.
\end{itemize}
{\bf Corank 3 Singularities}
\begin{itemize}
\item[] Through theorems {\bf 147--151}, $f\in K[[x,y,z]]$.
\item[{\bf 147.}]$j^2f(x,y,z)=0\Rightarrow $ one of the ten possibilities holds:

$\begin{array}{lllll}
j^3f&\sim_r & x^3+ax^2z+z^3+y^2z, a\neq 0&\mapsto &{\bf 148},\\
&\sim_r & x^3+axz^2+z^3+y^2z, a\neq 0&\mapsto &{\bf 149},\\
&\sim_r & x^3 + y^3 + xyz &\mapsto &{\bf 150},\\
&\sim_r & x^3+xyz &\mapsto &{\bf 150},\\
&\sim_r & xyz &\mapsto &{\bf 150},\\
&\sim_r & x^3+yz^2 &\mapsto &{\bf 150},\\
&\sim_r & x^2z+yz^2 &\mapsto &{\bf 150},\\
&\sim_r & x^3+xz^2 &\mapsto &{\bf 150},\\
&\sim_r & x^2y &\mapsto &{\bf 150},\\
&\sim_r & x^3 &\mapsto &{\bf 150},\\
&= & 0 &\mapsto &{\bf 151}.
\end{array}$
\item[{\bf 148.}] $j^3(f)=x^3+ax^2z+z^3+y^2z, a\neq 0\Rightarrow \mathrm{rmod}(f)\geq 4,\mu(f)\geq 11$.
\item[{\bf 149.}] $j^3(f)=x^3+axz^2+z^3+y^2z, a\neq 0\Rightarrow \mathrm{rmod}(f)\geq 4,\mu(f)\geq 11$.
\item[{\bf 150.}] $j^3(f)$ is degenerate $\Rightarrow \mathrm{rmod}(f)\geq 4,\mu(f)\geq 11$.
\item[{\bf 151.}] $j^3f=0\Rightarrow \mathrm{rmod}(f)\geq 6$.
\end{itemize}
{\bf Corank $> 3$ Singularities}
\begin{itemize}
\item[{\bf 152.}] $\mathrm{crk(f)}>3\Rightarrow \mathrm{rmod}(f)\geq 4$.
\end{itemize}
\section{Proof of the main results}\label{sec5}
We first use the splitting lemma to reduce the number of variables. Namely, if $f\in \mathfrak{m}^2\subset K[[{\bf x}]]$ has corank $\mathrm{crk}(f)=k\geq 0$, then 
$$f\sim_r g(x_1,\ldots, x_k)+Q(x_{k+1},\ldots,x_{n})$$ 
with $g\in \mathfrak{m}^3$ and $Q$ is a nondegenerate quadratic singularity (cf. \cite[Lemma 3.9, 3.12]{GN14}). One has moreover that $\mathrm{rmod}(f)$ in $K[[{\bf x}]]$ is equal to $\mathrm{rmod}(g)$ in $K[[x_1,\ldots, x_k]]$, cf. \cite[Lemma 3.11, 3.13]{GN14}.

Theorems {\bf 1}, {\bf 92}, {\bf 122}, {\bf 137} and {\bf 141} are obvious. Theorems {\bf 9}, {\bf 17}, {\bf 20}, {\bf 25}, {\bf 39}, {\bf 40}, {\bf 44}, {\bf 54}, {\bf 57}, {\bf 66}, {\bf 82}, {\bf 93}, {\bf 102}, {\bf 108}, {\bf 114} are proved by the Newton method \cite{New40} of a moving ruler (line, plane). This method reduces the proof to the counting of the integer points in triangles resp. polyhedrones on the exponent plane (resp. in the space). 

{\em Theorems concerning the geometrical classification problems}: The proofs of theorems {\bf 4}, {\bf 13}, {\bf 26}, {\bf 31}, {\bf 49}, {\bf 60}, {\bf 72}, {\bf 74}, {\bf 96}, {\bf 110}, {\bf 124}, {\bf 139}, {\bf 147} can be reduced to the classifications of orbits of the actions of some quasihomogenous diffeomorphism groups on the spaces of quasihomenous polynomials, see Section \ref{sec5.1} for a proof of Theorem {\bf 147}.

{\em Theorems on normal forms}: Theorems {\bf 2}, {\bf 123}, {\bf 138} follow from \cite[Thm 2.11]{Ng14}. The proofs of theorems {\bf 5}, {\bf 6}, {\bf 7}, {\bf 10}, {\bf 11}, {\bf 12}, {\bf 14}, {\bf 15}, {\bf 18}, {\bf 21}, {\bf 22}, {\bf 27}, {\bf 28}, {\bf 30}, {\bf 32}, {\bf 33}, {\bf 37}, {\bf 38}, {\bf 41}, {\bf 42},{\bf 43}, {\bf 45}, {\bf 46}, {\bf 50}, {\bf 51}, {\bf 55}, {\bf 56}, {\bf 58}, {\bf 61}, {\bf 62}, {\bf 64}, {\bf 67}, {\bf 68}, {\bf 72}, {\bf 75}, {\bf 76}, {\bf 77}, {\bf 79}, {\bf 80}, {\bf 81}, {\bf 83}, {\bf 84}, {\bf 85}, {\bf 88}, {\bf 89}, {\bf 94}, {\bf 95}, {\bf 97}, {\bf 98}, {\bf 100}, {\bf 103}, {\bf 104}, {\bf 107}, {\bf 109}, {\bf 111}, {\bf 112}, {\bf 115}, {\bf 125}, {\bf 127}, {\bf 132}, {\bf 142}, {\bf 143} are based on the techniques introduced in \cite{Arn72} and generalized in \cite{BGM12}, see Section \ref{sec5.2} for a proof of Theorem {\bf 14}. 

{\em Theorems on low bound of modality}: Theorems {\bf 121}, {\bf 151}, {\bf 152} are consequences of \cite[Prop. 2.18]{GN14}. Theorems {\bf 3}, {\bf 8}, {\bf 16}, {\bf 19}, {\bf 23}, {\bf 24}, {\bf 25}, {\bf 29}, {\bf 34}, {\bf 36}, {\bf 47}, {\bf 48}, {\bf 52}, {\bf 53}, {\bf 59}, {\bf 63}, {\bf 65}, {\bf 69}, {\bf 70}, {\bf 72}, {\bf 73}, {\bf 78}, {\bf 79}(iii), {\bf 80}(iii), {\bf 86}, {\bf 90}, {\bf 99}, {\bf 105}, {\bf 106}, {\bf 113}, {\bf 116--120}, {\bf 130}, {\bf 133}, {\bf 135}, {\bf 136}, {\bf 144}, {\bf 145}, {\bf 146}, {\bf 148}, {\bf 149}, {\bf 150} are proved by using the theory in \cite{GN14} (\cite{Ng13}), see Section \ref{sec5.3} for a proof of {\bf 25} and {\bf 70}.

{\em Theorems on adjacencies}: Theorem \ref{thm35} is proved inductively by applying Theorems {\bf 1}, {\bf 2}, {\bf 4--7}, {\bf 8--15}, {\bf 17}, {\bf 18}, {\bf 20}, {\bf 21}, {\bf 22}, {\bf 26}, {\bf 27}, {\bf 28}, {\bf 30--33}, {\bf 35}, {\bf 37--46}, {\bf 49--51}, {\bf 54--58}, {\bf 60--62}, {\bf 64}, {\bf 66--68}, {\bf 74--77}, {\bf 79--85}, {\bf 87--89}, {\bf 91--98}, {\bf 100--104}, {\bf 107--112}, {\bf 114--117}, {\bf 122--127}, {\bf 131}, {\bf 132}, {\bf 137--143}.

{\em Classification of unimodal and bimodal singularities (Theorems \ref{thm21}-\ref{thm24})}: Applying Theorems {\bf 1-153} and the spliting lemma (cf. \cite{GN14}) we obtain the list of families of singularities in Tables 1--11. The modularity of these families follows from simple caculations. To prove these singularities are unimodal resp. bimodal we use the theory of modality in \cite{GN14}. See Section \ref{sec5.5} for a proof that $E_{12}$ with $p>7$, is a class of unimodal singularities.



\subsection{Proof of Theorem {\bf 147}}\label{sec5.1} 
The theorem is obtained by combining the following lemmas (\ref{lm51}, \ref{lm52}, \ref{lm53}). Let $0\neq f\in K[x,y,z]$, with $\mathrm{char}(K)=3$, be a homogeneous polynomial of degree $3$. 
\begin{lemma}\label{lm51}
If $f$ is nonsingular, then $f$ is right equivalent to one of the following forms
$$x^3+ax^2z+z^3+y^2z,a\neq 0,\ x^3+axz^2+z^3+y^2z,a\neq 0.$$
\end{lemma}
\begin{proof}
cf. \cite[Chap. II, Prop.1.2]{Mil06}  
\end{proof}
\begin{lemma}\label{lm52}
If $f$ is singular in $\mathbb P^2_K$ and irreducible, then it is right equivalent to one of the following forms
$$x^3+y^3+xyz,\ x^3+y^2z.$$
\end{lemma}
\begin{proof}
Let $C$ be the curve in $\mathbb P^2$ defined by $f$. Take $P\in \mathrm{Sing}(C)$ and $P\neq Q\in C$. Let $L$ be the line in $\mathbb P^2$ connecting $P,Q$. Applying B\'ezout theorem we obtain that
\begin{eqnarray*}
3=\deg(C)\cdot \deg(L)\geq \mathrm{mt}_P(C)+\mathrm{mt}_Q(C).
\end{eqnarray*}
Hence $\mathrm{mt}_P(C)=2$ and $\mathrm{mt}_Q(C)=1$. We may assume $P=(0:0:1)$ and set $g(x,y):=f(x,y,1)$. Then $\mathrm{mt}(g)=2$ since $\mathrm{mt}_P(C)=2$. It yields that $g$ is right equivalent to one of the following forms
$$xy+h(x,y),\ y^2+h(x,y)$$
with $h(x,y)$ is a homogeneous polynomial of degree $3$. That is, $f$ is right equivalent to either
$$xyz+h(x,y) \text{ or } y^2z+h(x,y).$$
It hence follows by simple calculations that $f$ is right equivalent to one of the two forms
$$x^3+y^3+xyz,\ x^3+y^2z.$$
\end{proof}
\begin{lemma}\label{lm53}
If $f$ is reducible, then it is right equivalent to one of the following forms
$$x^3,\ x^2y,\ x^2z+yz^2,\ x^3+xyz,\ x^3+xz^2,\ xyz.$$
\end{lemma}
\begin{proof}
Let $f=g_1\cdot g_2$ with $\mathrm{mt}(g_1)=1,\ \mathrm{mt}(g_2)=2$. By the splitting lemma (cf. \cite{GN14})
$$g_2\sim_r ax^2+byz$$
with $a,b\in \{0,1\}$. That is $f\sim_r g_1\cdot (ax^2+byz)$. Consider the following cases:\\
$\bullet$ $a=1, b=0$: Then $f$ is right equivalent to $x^3$ or $x^2y$.\\
$\bullet$ $a=1, b=1$: Then $f\sim_r g_1\cdot (x^2+yz)$. Without loss of generality we may assume moreover that 
$$\{(0:1:0)\}\in \{g_1=0\}\cap \{x^2+yz=0\},$$ i.e. $g_1$ has the form $g_1=\alpha x+\beta z$.
\begin{itemize}
\item[-] If $\alpha = 0$, then $f\sim_r z(x^2+yz)$,
\item[-] if $\alpha\neq 0$, then $f\sim x(x^2+yz)$. 
\end{itemize}
$\bullet$ $a=0, b=1$: Then $f\sim_r g_1\cdot yz$. It yields that $f$ is right equivalent to one of the forms
$$y^2z,\ xyz,\ (y+z)yz.$$
Hence $f$ is right equivalent to one of the forms: $x^2y,\ xyz,\ x^3+xz^2$.
\end{proof}
\subsection{Proof of Theorem ${\bf 14}$}\label{sec5.2}
Let $f\in K[[x,y]]$ with $p=\mathrm{char}(K)=5$ and $j_{x^3,y^6}f=x^3+bx^2y^2+y^6+ay^5$. We will show that 
$f$ is right equivalent to $f_0:=x^3+bx^2y^2+y^6+ay^5$, i.e. $f$ is of type $J_{2,0}$.

In fact, put $g:=f-ay^5$, then $j_{x^3,y^6}g=x^3+bx^2y^2+y^6$. Applying \cite[Thm. 4.4]{BGM11} we obtain that 
$g\sim_r x^3+bx^2y^2+y^6$. We can see moreover that there exists a coordinate change of the form
$$x\mapsto x+\varphi_1(x,y), y\mapsto y+\varphi_2(x,y)$$
with $\mathrm{mt}(\varphi_i)\geq 2$ such that 
$$g(x+\varphi_1,y+\varphi_2)=x^3+bx^2y^2+y^6.$$
It yields
$$f_1:=f(x+\varphi_1,y+\varphi_2)=x^3+bx^2y^2+y^6+a (y+\varphi_2)^5=x^3+bx^2y^2+y^6+ay^5+a\varphi_2^5.$$ 
It is easy to see that $\mathfrak{m}^7\subset \mathfrak{m}^2\cdot j(f_1)$. By \cite[Thm. 2.1]{BGM12}, $f_1$ is right 9-determined and hence $f_1\sim_r f_0$ since $\mathrm{mt}(\varphi_2^5)\geq 10$. This completes the proof.
\subsection{Proof of theorems on lower bound of modality}\label{sec5.3}
For the proof of these theorems we need the following lemma which is deduced from Corollaries A.4, A.9, A.10 of \cite{GN14} (see \cite[Prop. 3.2.4, Cor. 3.3.4 and Cor. 3.3.6]{Ng13} for more details).
\begin{lemma}\label{lm5.3.1}
Let the algebraic groups $G$ resp. $G'$ act on the varieties $X$ resp. $X'$. Let 
$h: Y\to X$ a morphism  of varieties and let $h':Y\to X'$ an open morphism such that 
\begin{equation}\label{eq5.3.1}
h^{-1}(G\cdot h(y))\subset h'^{-1}(G'\cdot h' (y)), \forall y\in Y.
\end{equation}
Then for all $y\in Y$ we have
$$G\text{-}\mathrm{mod}(h(y))\geq G'\text{-}\mathrm{mod}(h'(y))\geq \dim X'-\dim G'.$$
\end{lemma}
\begin{lemma}\label{lm5.3.2}
Let $f\in K[[x,y]]$ with $\mathrm{char}(K)>3$. Then $\mathrm{rmod}(f)\geq 2+l$ with $l\geq 0$, if either
\begin{itemize}
\item[(i)] $f\in \langle x,y^{3+l}\rangle^3$; or
\item[(ii)] $f\in \langle x^2,y^{3+l}\rangle^2$.
\end{itemize}
\end{lemma}
\begin{proof}
We prove only for (i) since the proof for (ii) is similar. Let $k$ be sufficiently large for $f$, i.e. $\mathrm{rmod}(f)=\mathcal{R}_k\text{-}\mathrm{mod}(f)$. We denote
$$\Delta:=\{(3;s),(2;3+s),(1;6+s),(0;9+s)\ |\ 0\leq s\leq l+1\}\subset \mathbb N^2,$$  
$$\Delta_1:=\{(1;s),(0;3+l+s)\ |\ 0\leq s\leq l+1\} \text{ and }\Delta_2:=\{(0;1+s)\ |\ 0\leq s\leq l+1\}\}\subset \mathbb N^2$$
and define
$$X:=\{\sum_{(i,j)\in \Delta} a_{i,j}x^iy^j\in K[[x,y]]\ |\ a_{i,j}\in K\}\cong \mathbb A^{4(l+2)},$$
$$G:=X_1\times X_2\cong \mathbb A^{3(l+2)},$$
where
$$X_1:=\{\sum_{(i,j)\in \Delta_1} a_{i,j}x^iy^j\in K[[x,y]]\ |\ a_{i,j}\in K,a_{10}\neq 0\},$$
$$ X_2:=\{\sum_{(i,j)\in \Delta_2} b_{i,j}x^iy^j\in K[[x,y]]\ |\ b_{i,j}\in K,b_{01}\neq 0\}.$$
Using the projections
$$\pi_1\colon J_k\to X_1, \sum_{(i,j)} a_{i,j}x^iy^j\mapsto \sum_{(i,j)\in \Delta_1} a_{i,j}x^iy^j,$$
$$\pi_2\colon J_k\to X_2, \sum_{(i,j)} a_{i,j}x^iy^j\mapsto \sum_{(i,j)\in \Delta_2} a_{i,j}x^iy^j,$$
$$\pi\colon J_k\to X, \sum_{(i,j)} a_{i,j}x^iy^j\mapsto \sum_{(i,j)\in \Delta} a_{i,j}x^iy^j$$
and
\begin{eqnarray*}
\bar\pi\colon \mathcal R_k&\to& G=X_1\times X_2\\
\Phi=(\Phi_1,\Phi_2)&\mapsto& \big(\pi_1(\Phi_1),\pi_2(\Phi_2)\big)
\end{eqnarray*}
we may define a multiplication on $G$, resp. an action map of $G$ on $X$ as follows
\begin{eqnarray*} 
\bullet\colon G\times G &\to& G\\
(\phi,\phi')&\mapsto & \bar\pi(\phi\circ \phi'),
\end{eqnarray*}
resp.
\begin{eqnarray*}
G\times X &\to& X\\
(\phi,g)&\mapsto & \pi\left(\phi(g)\right).
\end{eqnarray*}
By a simple calculation we can verify that the morphisms
$\iota\colon Y:=\langle x,y^{3+l}\rangle^3/\mathfrak{m}^{k+1}\hookrightarrow J_k$ and $\pi\colon Y\to X$ satisfy 
$$\iota^{-1}(\mathcal{R}_k\cdot \iota(g))\subset \pi^{-1}(G\cdot \pi (g)), \forall g\in Y.$$
Hence applying Lemma \ref{lm5.3.1} we obtain that
$$\mathrm{rmod}(f)=\mathcal{R}_k\text{-}\mathrm{mod}(\iota(f))\geq \dim X-\dim G=2+l.$$
\end{proof}
\subsection{Computing the modality of $E_{12}$}\label{sec5.5}
We shall show that $E_{12}$ is a class of unimodal singularities. To compute the modality of a singularity we use the general argument in \cite{GN14}, in particular, the following lemma.
\begin{lemma}\label{lm5.5.1}
Assume that $f\in K[[{\bf x}]]$ deforms only into finitely many families $h^{(i)}_{t}({\bf x})$ over varieties $T^{(i)},i\in I$. Then
$$\mathrm{rmod}(f)\leq \max_{i\in I} \dim T^{(i)}.$$
Assume further that the families $h^{(i)}_{t}({\bf x})$ are all modular. Then 
$$\mathrm{rmod}(f)= \max_{i\in I} \dim T^{(i)}.$$
\end{lemma}
\begin{proof}
cf. \cite{GN14}
\end{proof}
\begin{proof}[Proof for $E_{12}$]
Assume that $f=x^3+y^7+axy^5\in K[[x,y]]$ with $p=\mathrm{char}(K)>7$ and $a\in K$, is of type $E_{12}$. We will show that 
$$\mathrm{rmod}(f)=1.$$ 
In fact, by Theorem \ref{thm35} (or, Theorems {\bf 1}--{\bf 9}, {\bf 26}, {\bf 27}, {\bf 28}), $f$ deforms only into the following modular families
$$E_{12},\ A_{k} (k\leq 6), D_{k} (k\leq 8), E_{6}, E_{7}, E_{8}, J_{2,0}, J_{2,1}.$$
Hence it follows from Lemma \ref{lm5.5.1} that $f$ is right unimodal singularities.
\end{proof}


\begin{thebibliography}{99}
\baselineskip=16pt
\bibitem{AGV85} V. I. Arnol'd, S. M. Gusein-Zade, and A.N. Varchenko, {\em Singularities of differentiable maps,} Vol I. Birkh\" auser (1985).

\bibitem{Arn72} V. I. Arnol'd, {\em Normal forms for functions near degenerate critical points, the Weyl groups of $A_k, D_k, E_k$ and Lagrangian singularities,} Functional Anal. Appl. 6 (1972), 254-272.

\bibitem{Arn73} V. I. Arnol'd, {\em Classification of unimodal critical points of functions,} Functional Anal. Appl. 7 (1973), 230-231.

\bibitem{Arn74} V. I. Arnol'd, {\em Normal forms of functions in the neighbourhoods of degenerate critical points,} Russian Math. surveys 29 (1974), no. 2, 11-49.

\bibitem{Arn76} V. I. Arnol'd, {\em Local normal form of functions,} Invent. Math. 35 (1976), 87--109.

\bibitem{Art77} M. Artin, {\em Coverings of the rational double points in characteristic $p$,} in: Complex Analysis and Algebraic Geometry, ed. W. L. Baily, jr. and T. Shioda, Iwanami Shoten, Publ., Cambridge Univ. Press, 1977.

\bibitem{BGM11} Y. Boubakri, G.-M. Greuel, and T. Markwig, {\em Normal forms of hypersurface singularities in positive characteristic,} Mosc. Math. J. 11(2011), no. 4, 657--683.

\bibitem{BGM12} Y. Boubakri, G.-M. Greuel, and T. Markwig, {\em Invariants of hypersurface singularities in positive characteristic,} Rev. Math. Complut. 25(2012), no. 1,61--85.



\bibitem{Gab74}  A. M. Gabrielov, {\em Bifurcations, Dynkin diagrams and the modality of isolated singularities,} Funktsional. Anal. i Prilozhen. 8:2, (1974), 7--12.\\(Engl. translation: Funct. Anal. Appl. 8 (1974), 94--98.)

\bibitem{Giu77} M. Giusti, {\em Classification des Singulariti\'es isol\'ees d'intersectios comple\`etes simples,} C. R. Acad. Sci. Paries S\'er. A-B 284(1977), no. 3, A167--A170.

\bibitem{GK90} G.-M. Greuel, H. Kr\"oning, {\em Simple singularities in positive characteristic,} Math. Z. 203 (1990), 339-354.

\bibitem{GLS06} G.-M. Greuel, C. Lossen and E. Shustin, {\em Introduction to Singularities and deformations,} Math. Monographs, Springer-Verlag (2006).


\bibitem{GN14} G.-M. Greuel, H. D. Nguyen, {\em Right simple singularities in positive characteristic,} to appear Journal f\"ur die Rein und Angewand. 

\bibitem{KaS72} A. Kas, M. Schlessinger, {\em On the versal deformation of a complex space with an isolated singularity,} Math. Ann. 196 (1972), 23-29.

\bibitem{Kou76} A.G. Kouchnirenko, {\em Poly\`edres de Newton et nombres de Milnor,} Invent. Math. 32 (1976), 1-31.

\bibitem{Le71} D.T. L\^e : Th\`ese de Doctorat,  Paris  VII,  D\'ec. 1971 

\bibitem{Lue87} I. Luengo, {\em The $\mu$-constant stratum is not smooth,} Invent. Math., 90 (1987), 13--92.

\bibitem{Mil06} J. S. Milne, {\em Elliptic Curves,} BookSurge Publ., 2006.

\bibitem{Mil68} J. Milnor, {\em Singular points of complex hypersurfaces,} Princeton Univ. Press (1968).

\bibitem{MHW01} A. Melle-Hern\'andez, C. T. C. Wall, \emph{Pencils of curves on smooth surfaces}, Proc. Lond. Math. Soc., III. Ser. 83 (2001), no.~2, 257--278.
  
\bibitem{New40} I. Newton, {\em La m\'ethode des fluxions}, trad. Buffon. Paris: Debure l'Ain6 (1740).

\bibitem{Ng13} H. D. Nguyen, {\em Classification of singularities in positive characteristic,} Ph.D. thesis, TU Kaiserslautern (2013).

\bibitem{Ng14} H. D. Nguyen, {\em The right classification of univariate singularities in positive characteristic,} J. Singularities 10 (2014), 235--249.

\bibitem{SY79} M. Suzuki, E. Yoshinaga, {\em Normal forms  of non-degenerate quasihomogeneous functions  with inner  modality $\leq 4$}, Invent. Math. 55 (1979), 185--206.

\bibitem{Var82} A. N. Varchenko, {\em A lower bound for the codimension of the stratum $\mu$=const in terms of the mixed Hodge structure,} Moscow Univ. Math. Bull., {\bf 37}(1982), no. 6, 30--33.


\bibitem{Wal83} C. T. C. Wall, {\em Classification of unimodal isolated singularities of complete intersections,} pp 625--640 in Proc. Symp. in Pure Math. 40ii (Singularities) (ed. P. Orlik) Amer. Math. Soc., 1983.


\end{thebibliography}
\end{document}